\documentclass{amsart}

\usepackage[T1]{fontenc}
\usepackage[english]{babel}
\usepackage{csquotes}
\usepackage{mathalpha}
\usepackage{xparse}
\usepackage{amssymb}
\usepackage{amscd}
\usepackage{mathtools}
\usepackage[all,cmtip]{xy}

\usepackage[backend=biber,style=alphabetic,sorting=nyt,doi=false,url=false,isbn=false]{biblatex}
\usepackage[hidelinks]{hyperref}
\usepackage{faktor}
\usepackage[disable]{todonotes} 
\usepackage[shortcuts]{extdash} 
\usepackage{soul} 

\makeatletter
\renewcommand\subsection
{
\@startsection{subsection}{2}
\z@{0.5\linespacing\@plus0.7\linespacing}{0.5\linespacing}{\normalfont\bfseries}
}
\makeatother

\usepackage{enumitem}
\setlist[itemize]{leftmargin=1em}

\makeatletter
\patchcmd{\@setaddresses}{\indent}{\noindent}{}{}
\patchcmd{\@setaddresses}{\indent}{\noindent}{}{}
\patchcmd{\@setaddresses}{\indent}{\noindent}{}{}
\patchcmd{\@setaddresses}{\indent}{\noindent}{}{}
\makeatother

\numberwithin{equation}{section}

\theoremstyle{plain}
\newtheorem{fact}[equation]{Fact}

\newtheorem{thm}[equation]{Theorem}
\newtheorem{lem}[equation]{Lemma}

\newtheorem*{prop*}{Proposition}
\newtheorem*{thm*}{Theorem}
\newtheorem*{lem*}{Lemma}
\newtheorem*{cor*}{Corollary}

\theoremstyle{definition}
\newtheorem{defn}[equation]{Definition}

\newtheorem*{defn*}{Definition}
\newtheorem*{question*}{Question}

\theoremstyle{remark}
\newtheorem{remark}[equation]{Remark}

\newtheorem*{remark*}{Remark}
\newtheorem*{example*}{Example}
\newtheorem*{notation*}{Notation}
\newtheorem*{terminology*}{Terminology}

\NewDocumentCommand\anbr{smo}{\IfBooleanTF{#1}{\IfNoValueTF{#3}{\left\langle#2\right\rangle}{\left\langle#2\colon\,#3\right\rangle}}{\IfNoValueTF{#3}{\left\langle#2\right\rangle}{\left\langle#2\right\rangle_{#3}}}} 
\newcommand{\bijto}{\to[bij]}
\NewDocumentCommand\tuple{m}{\langle #1\rangle}
\NewDocumentCommand\cubr{smo}{\IfBooleanTF{#1}{\IfNoValueTF{#3}{\left\{#2\right\}}{\left\{#2\colon\,#3\right\}}}{\IfNoValueTF{#3}{\left\{#2\right\}}{\left\{#2\right\}_{#3}}}} 

\newcommand{\intps}{\mathrel{\rhd}}
\newcommand{\intpd}{\mathrel{\lhd}}
\newcommand{\isoto}{\to[iso]}
\newcommand{\leqs}{\mathrel{\leqslant}}
\newcommand{\ls}{\mathrel{<}}
\let\models\relax
\newcommand{\models}{\mathrel{\vDash}}

\NewDocumentCommand\mt{m}{\mathtt{#1}}

\newcommand{\proves}{\mathrel{\vdash}}

\NewDocumentCommand\restr{m}{\hspace{-0.3em}\restriction_{#1}}
\NewDocumentCommand\robr{smo}{\IfBooleanTF{#1}{\IfNoValueTF{#3}{\left(#2\right)}{\left(#2\colon\,#3\right)}}{\IfNoValueTF{#3}{\left(#2\right)}{\left(#2\right)_{#3}}}} 
\NewDocumentCommand\sqbr{smo}{\IfBooleanTF{#1}{\IfNoValueTF{#3}{\left[#2\right]}{\left[#2\colon\,#3\right]}}{\IfNoValueTF{#3}{\left[#2\right]}{\left[#2\right]_{#3}}}} 
\newcommand{\strongs}{\mathrel{\ooalign{$\vdash$\cr$\vDash$}}}
\let\to\relax
\NewDocumentCommand\to{o}{\mathrel{\IfNoValueTF{#1}{\xrightarrow{}}{\xrightarrow{\scriptscriptstyle\mathrm{#1}}}}}
\NewDocumentCommand\tup{m}{\overline{#1}}

\newcommand{\AC}{\mathrm{AC}}
\newcommand{\ACA}{\mathrm{ACA}}
\newcommand{\ACAz}{\mathrm{ACA}_{0}}
\newcommand{\AS}{\mathrm{AS}}

\NewDocumentCommand\BS{m}{\mathrm{B\Sigma}_{#1}}

\newcommand{\Com}{\mathrm{Com}}

\NewDocumentCommand\Def{m}{\ensuremath{\mathrm{Def}(#1)}}
\NewDocumentCommand\Defpf{m}{\ensuremath{\mathrm{Def^{pf}}(#1)}}

\NewDocumentCommand\dom{o}{\IfNoValueTF{#1}{\mathrm{dom}}{\mathrm{dom}(#1)}}
\NewDocumentCommand\eeq{m}{\mathrm{eeq_{#1}}}
\newcommand{\GB}{\mathrm{GB}}

\newcommand{\id}{\mathrm{id}}
\newcommand{\Ind}{\mathrm{Ind}}
\NewDocumentCommand\IS{m}{\mathrm{I\Sigma}_{#1}}
\NewDocumentCommand\hf{m}{\mathrm{HF}(#1)}
\newcommand{\iso}{\mathrm{iso}}

\NewDocumentCommand\Lang{o}{\IfNoValueTF{#1}{\calL}{\calL_{#1}}}

\newcommand{\Ord}{\mathrm{Ord}}
\newcommand{\PA}{\mathrm{PA}}
\newcommand{\PAm}{\mathrm{PA^{-}}}
\NewDocumentCommand\PC{m}{\mathrm{PC}[#1]}

\newcommand{\rank}{\mathrm{rank}}

\newcommand{\Repl}{\mathrm{Repl}}
\NewDocumentCommand\Tarski{m}{#1\text{-}\mathrm{Tarski}}

\newcommand{\tc}{\mathrm{tc}}

\NewDocumentCommand\SAT{m}{#1\text{-}\mathrm{SAT}}
\newcommand{\Sat}{\mathrm{Sat}}

\NewDocumentCommand\sPC{m}{\mathrm{sPC}[#1]}

\newcommand{\ZF}{\mathrm{ZF}}
\newcommand{\ZFbase}{\mathrm{ZF_{base}}}
\newcommand{\ZFC}{\mathrm{ZFC}}

\newcommand{\Zt}{\mathrm{Z}_{2}}

\newcommand{\bbN}{\mathbb{N}}

\newcommand{\bbZ}{\mathbb{Z}}

\newcommand{\calA}{\mathcal{A}}

\newcommand{\calI}{\mathcal{I}}

\newcommand{\calL}{\mathcal{L}}
\newcommand{\calM}{\mathcal{M}}
\newcommand{\calN}{\mathcal{N}}

\newcommand{\calP}{\mathcal{P}}

\newcommand{\calX}{\mathcal{X}}

\newcommand{\frakA}{\mathfrak{A}}

\newcommand{\frakM}{\mathfrak{M}}
\newcommand{\frakN}{\mathfrak{N}}

\newcommand{\rmI}{\mathrm{I}}

\newcommand{\ttM}{\mathtt{M}}
\newcommand{\ttN}{\mathtt{N}}

\begin{document}


\title[Definiteness properties of first-order schemes]{Definiteness properties\\ of first-order schemes}

\author{Piotr Gruza}
\address[Piotr Gruza]{University of Warsaw, Faculty of Mathematics, Informatics and Mechanics}
\email{p.gruza3@uw.edu.pl}

\author{Mateusz Łełyk}
\address[Mateusz Łełyk]{University of Warsaw, Faculty of Philosophy}
\email{mlelyk@uw.edu.pl}
\thanks{The work of the second author was supported by the National Science Centre (NCN) grant no. 2022/46/E/HS1/00452.}



\begin{abstract}

The paper aims to establish a convenient formal framework for investigating the phenomenon of scheme definiteness, exemplified by first-order internal categoricity as studied by V\"a\"an\"anen, among others. To this end, we introduce the notion of \textit{$\Phi$-definiteness}, thereby refining and extending the conceptual landscape that underlies various first-order categoricity notions in the literature (internal categoricity, strong internal categoricity, intolerance). We provide arguments for the robustness of our definition and present examples of schemes that separate different categoricity- and completeness-like notions. Finally, we offer a brief glimpse into the issue of the definiteness of two canonical foundational schemes -- the induction scheme and the replacement scheme.

\end{abstract}

\maketitle



\section*{Introduction}

It is widely believed, not only among philosophers but also mathematicians, that sentences which express facts about certain abstract domains have a determinate truth value: they are either true or false. Trivial examples aside, one of the most important cases is provided by the natural numbers. Even though our best mathematical theories are incomplete with respect to sentences about non-negative integers (if only because of the G\"odel's incompleteness results), we tend to believe that these are rather limitations of our current axiomatic framework, then indicators that some such statements may simply lack the truth value. This conviction is often backed by categoricity arguments: if we can show that our chosen axioms uniquely determine a certain structure (up to an isomorphism), then this concrete structure determines in turn the truth values of all the sentences of the chosen language. In particular, in the context of arithmetic the conclusion of the Dedekind categoricity argument, is that each two full models of arithmetic and the second-order induction axiom (i.e. models in which all subsets of the universe satisfy the induction principle) are isomorphic. In the context of set theory we have a similar result due to Zermelo for the second order replacement scheme: each two full models of second-order replacement scheme of the same height (i.e. whose ordinals are isomorphic as well-orders), are isomorphic. However, as discussed at length in \cite{Button_Walsh_book}, p.159-160 and \cite{Button_MIR}), these arguments do not deliver the goods: establishing these results, especially in the case of arithmetic, requires the grasp of a much richer structure (all subsets of the natural numbers) than the one being discussed. This is where \textit{internal categoricity results} kick in. Roughly speaking, the goal is to obtain the analogues of Dedekind and Zermelo categoricity results for Peano Arithmetic (PA) and Zermelo-Fraenkel Set Theory (ZF) respectively in axiomatic frameworks not requiring full second-order semantics (see \cite{Feferman_Hellman}, \cite{Vaananen_Wang_Notre_Dame}, \cite{Vaananen_Zermelo_ext}). These frameworks are provided by the full second-order comprehension scheme (treated axiomatically as a deductive system, as in \cite{Vaananen_Zermelo_ext}) or even weak first-order arithmetical theory (such as is \cite{Feferman_Hellman}). We shall refer to the system in which the respective categoricity (later on: definiteness) claim for a theory $S$ is established as the categoricity base for $S$. The role of these formal results in settling the philosophical problem of determinateness of statements of arithmetic and set theory has become an object of vivid discussion (see \cite{Button_MIR}, \cite{Maddy_Vaananen}, \cite{Fischer_zicchetti}, \cite{Waxman_Piccolo}). 

\bigskip

The line of research focused rather on the internal categoricity of concrete, most prominent foundational systems, namely PA and ZF. The first attempt at a general definition of the property "S is internally categorical" was taken in \cite{ena_lel}, which naturally led to the question: what kind of object should $S$ be? Results of of V\"a\"ananen directly applied to the scheme of induction (conjuncted with finitely many basic axioms for successor, addition and multiplication) and the scheme of replacement (likewise conjuncted with finitely many axioms expressing basic properties of $\in$), so internal categoricity was naturally defined as a  property of schemes, to wit: first-order sentences with a free second-order variable. It was then observed that this notion does not lift so easily to a property of first-order theories, seen as (deductively closed) sets of first-order sentences. As a concrete example: there is a scheme $\sigma$ such that the theory consisting of arithmetical substitutional instances of $\sigma$ with arithmetical formulae is deductively equivalent to PA, but $\sigma$ is \emph{not} internally categorical. \cite{ena_lel}'s reaction was that instead of trying to adapt the notion of internal categoricity to theories via some more complicated relations, one should develop metamathematics of schemes, treated as a self-standing foundational objects. This resulted in the definition of a scheme accomodation in a model, a strong model of a scheme and the adaptation of the notion of solidity to the scheme context.

\bigskip

The main ambition of the current paper is to make further steps in the direction uncovered by \cite{ena_lel}. We want to establish a convenient formal framework for investigating the phenomenon of scheme definiteness, exemplified by first-order internal categoricity as studied by V\"a\"an\"anen, among others. To this end, we introduce the notion of \textit{$\Phi$-definiteness}, thereby refining and extending the conceptual landscape that underlies various first-order categoricity notions in the literature (internal categoricity, strong internal categoricity, intolerance). We diverge from the definitions given in \cite{ena_lel} in two different ways. First, we explicitly introduce a parameter to serve as a definiteness base, which enables us to examine how the metatheories (or metaschemes) interact with different notions of definiteness. We allow an arbitrary sequential theory to be used in the role of such a definiteness base, which enables us to give definiteness arguments in very weak theories. Secondly, instead of introducing somewhat complicated language extensions with fresh predicate, we use a much better known extensions of theories with predicative comprehension. This allows us to compactify the definitions and make them better integrated with the standard metamathematical practice. These machinery, along with an adaptation of the notion of interpretability to the context of schemes, is introduced in Section 2. In the next section, we collect various fundamental metamathematical properties of our definiteness criteria. We show that they are invariant under the appropriate notion of bi-interpretability (Section 3.1). Furthermore (in Section 3.2) we show how to separate various definiteness criteria with sequential schemes. Last but not least, in Section 3.3, we investigate whether a given definiteness criterion is preserved upon extending the given scheme in a richer language, for example by adding a new set-sort and full comprehension for it. We conclude our paper with results about definiteness of $\PA$, $\ZF$ and its proper subsystems.


\section{Preliminaries}
In this section, we set out basic definitions and propositions that underlie our work, concentrating on the notions of interpretation and sequentiality. This presentation is not intended to be in any way comprehensive, and we allow ourselves a certain degree of imprecision; we aim rather to fix terminology and notation, and to explain some of the adopted conventions. For a more detailed account, we refer the reader to the preliminaries of \cite{visser_sivs}.

\bigskip
In this article, all theories are one-sorted first-order. Whenever we write about a second-order theory\footnote{For notational simplicity, we assume that classes are extensional; this has no impact on the results.}, we mean a one-sorted first-order theory encoding it, with first- and second-order objects indicated by predicates of disjoint extensions. Furthermore, our attention is restricted to finite, purely relational languages, with function symbols used only as shorthand. The finiteness assumption is important, whereas the latter restriction is made merely for convenience.

\begin{defn}\label{defn_translation}
A function $\mathtt{t}\colon\Lang[T]\to\Lang[S]$ is a translation of $\Lang[T]$ to $\Lang[S]$ if there exist $\Lang[S]$-formulas $\delta(\tup{x}),\eta(\tup{x}_{1},\tup{x}_{2})$, and $\varphi_{R}(\tup{x}_{1},\ldots,\tup{x}_{\mathrm{arity}(R)})$ for each symbol $R$ from the signature of $\Lang[T]$, with $\mathrm{length}(\tup{x}_{i})=\mathrm{length}(\tup{x})$ for each $i$, such that:
\begin{itemize}
\smallskip
\item $\proves(\forall\tup{x}_{1},\tup{x}_{2})(\eta(\tup{x}_{1},\tup{x}_{2})\rightarrow\delta(\tup{x}_{1})\land\delta(\tup{x}_{2}))$;
\smallskip\smallskip
\item $\proves(\forall\tup{x}_{1},\ldots,\tup{x}_{\mathrm{arity}(R)})(\varphi_{R}(\tup{x}_{1},\ldots,\tup{x}_{\mathrm{arity}(R)})\rightarrow\delta(\tup{x}_{1})\land\ldots\land\delta(\tup{x}_{\mathrm{arity}(R)}))$, for each symbol $R$ from the signature of $T$;
\end{itemize}
and, up to variable renaming:
\begin{itemize}
\smallskip
\item $\mathtt{t}(v_{1}=v_{2})=\eta(\tup{x}_{1},\tup{x}_{2})$;
\smallskip\smallskip
\item $\mathtt{t}(R(v_{1},\ldots,v_{\mathrm{arity}(R)}))=\varphi_{R}(\tup{x}_{1},\ldots,\tup{x}_{\mathrm{arity}(R)})$;
\smallskip\smallskip
\item $\mathtt{t}\mathtt(\lnot\alpha)=\lnot\mathtt{t}(\alpha)$ and $\mathtt{t}\mathtt(\alpha\land\beta)=\mathtt{t}(\alpha)\land\mathtt{t}(\beta)$, for any $\alpha,\beta\in\Lang[T]$;
\smallskip\smallskip
\item $\mathtt{t}\mathtt((\exists v)\alpha(v))=(\exists\tup{x})(\delta(\tup{x})\land\mathtt{t}(\alpha)(\tup{x}))$, for any $\alpha\in\Lang[T]$.
\end{itemize}
\end{defn}

\begin{defn}\label{defn_interpretation}
A translation $\mathtt{t}\colon\Lang[T]\to\Lang[S]$ is called an interpretation of a theory $T$ in a theory $S$, denoted with $\mathtt{t}\colon S\intps T$ or $\mathtt{t}\colon T\intpd S$,  if for every sentence $\sigma\in\Lang[T]$, if $T\proves\sigma$, then $S\proves\mathtt{t}(\sigma)$.
\end{defn}

If we are concerned only with the language of $T$, we may also write \mbox{$\mathtt{t}\colon S\intps\Lang[T]$} or $\mathtt{t}\colon\Lang[T]\intpd S$. The theory $S$ is called the ground theory of the interpretation. We identify two interpretations $\mathtt{t}$ and $\mathtt{t}'$ if they are generated by $\delta$, $\eta$, $\varphi_{R}$ and $\delta'$, $\eta'$, $\varphi_{R}'$, respectively, such that the ground theory proves that both formulas in each corresponding pair are equivalent. To improve readability, we use $\dom^{\mathtt{t}}$, $=^{\mathtt{t}}$, $R^{\mathtt{t}}$ instead of $\delta$, $\eta$, $\varphi_{R}$.

\bigskip
Sometimes, we employ the phrase ``the ground structure'', by which we mean the virtual structure described by the ground theory via its primitive relations from the signature. We often denote the ground structure with $V$ by analogy with the universe of all sets. An~analogous object obtained by applying $\mathtt{t}$ in $S$ is denoted by $\mathtt{t}^{S}$, or simply by $\mathtt{t}$ when it does not lead to confusion. We use this terminology to describe structural properties of interpretations, as formalized in the ground theory. For example, we may write: ``For every sentence $\sigma$, it is provable in $S$ that the ground structure satisfies $\sigma$ if and only if $\mathtt{t}^{S}$ satisfies $\sigma$'' (which is the English counterpart of $(\forall \sigma)(S\proves\sigma\leftrightarrow\mathtt{t}(\sigma))$).

\begin{defn}
Let $\mathtt{t}$ be a translation with $\delta(\tup{x})$ and $\eta(\tup{x}_{1},\tup{x}_{2})$ as in Definition \ref{defn_translation}. Then $\mathtt{t}$ is:
\begin{itemize}
\item $n$-dimensional if $n=\mathrm{length}(\tup{x})$;
\smallskip
\item unrelativized if $\delta(\tup{x})$ is $\tup{x}=\tup{x}$;
\smallskip
\item identity-preserving if $\eta(\tup{x}_{1},\tup{x}_{2})$ is $\tup{x}_{1}=\tup{x}_{2}$;
\smallskip
\item direct if it is one-dimensional, unrelativized, and identity-preserving.
\end{itemize}
\end{defn}

We use similar terminology for interpretations, but require only ground-theory equivalence of formulas instead of their literal equality in the second and third bullet points. Note that in the case of interpretations, none of the above properties depends on the choice of $\delta$ or $\eta$.

\begin{fact}
\ 

\textbf{a)}\ For any translations $\mathtt{t}\colon\Lang[S]\intps\Lang[T]$ and $\mathtt{u}\colon\Lang[T]\intps\Lang[U]$, their composition $\mathtt{u}\circ\mathtt{t}$ (also denoted by $\mathtt{t}\mathtt{s}$) is a translation of $\Lang[U]$ to $\Lang[S]$.

\smallskip
\textbf{b)}\ For any interpretions $\mathtt{t}\colon S\intps T$ and $\mathtt{u}\colon T\intps U$, their composition $\mathtt{u}\circ\mathtt{t}$ (also denoted by $\mathtt{t}\mathtt{s}$) is an interpretation of $U$ in $S$.
\end{fact}

The next definition introduces a weaker notion of sameness of interpretations than the convention adopted below Definition \ref{defn_interpretation}. It is intended to capture the property: ``There exists an $\Lang[S]$-formula $\iota$ that, provably in $S$, defines an isomorphism between $\mathtt{t}^{S}$ and $(\mathtt{t}^{\ast})^{S}$''.  

\begin{defn}
Interpretations $\mathtt{t},\mathtt{t}^{\ast}\colon S\intps T$ are isomorphic if there exists an $\Lang[S]$-formula $\iota(\tup{x},\tup{x}^{\ast})$\footnote{The lengths of $\tup{x}$ and $\tup{x}^{\ast}$ may be different.} such that $S$ proves the universal closures of the following formulas\footnote{The second, third, and fourth sentences say that $\iota$ determines a bijection between $\eta$ equivalence classes on $\delta$ and $\eta^{\ast}$ equivalence classes on $\delta^{\ast}$. And the last ones say that this bijection preserves all relations of the signature of $T$.}:
\begin{itemize}
\smallskip
\item $\iota(\tup{x},\tup{x}^{\ast})\rightarrow\dom^{\mathtt{t}}(\tup{x})\land\dom^{\mathtt{t^{\ast}}}(\tup{x}^{\ast})$;
\smallskip\smallskip
\item $\dom^{\mathtt{t}}(\tup{x})\rightarrow(\exists\tup{x}^{\ast})(\dom^{\mathtt{t}^{\ast}}(\tup{x}^{\ast})\land\iota(\tup{x},\tup{x}^{\ast}))$;
\smallskip\smallskip
\item $\dom^{\mathtt{t}^{\ast}}(\tup{x}^{\ast})\rightarrow(\exists\tup{x})(\dom^{\mathtt{t}}(\tup{x})\land\iota(\tup{x},\tup{x}^{\ast}))$;
\smallskip\smallskip
\item $\iota(\tup{x},\tup{x}^{\ast})\rightarrow(\tup{x}=^{\mathtt{t}}\tup{y}\leftrightarrow\iota(\tup{y},\tup{x}^{\ast}))\land(\tup{x}^{\ast}=^{\mathtt{t}^{\ast}}\tup{z}^{\ast}\leftrightarrow\iota(\tup{x},\tup{z}^{\ast}))$;
\smallskip\smallskip
\item $\bigwedge_{i\leqs\mathrm{arity}(R)}\iota(\tup{x},\tup{x}^{\ast})\rightarrow(R^{\mathtt{t}}(\tup{x}_{1},\ldots,\tup{x}_{\mathrm{arity}(R)})\leftrightarrow R^{\mathtt{t^{\ast}}}(\tup{x}^{\ast}_{1},\ldots,\tup{x}^{\ast}_{\mathrm{arity}(R)}))$, for each symbol $R$ from the signature of $T$.
\end{itemize}
\end{defn}

\begin{defn}
Adjunctive set theory ($\AS$) is the theory ($\subseteq\Lang[\in]$) axiomatized by the following two axioms:
\begin{itemize}
\smallskip
\item $(\exists a)(\forall x)(x\notin a)$,
\smallskip\smallskip
\item $(\forall a)(\forall b)(\exists c)(\forall x)(x\in c\leftrightarrow x\in a\lor x=b)$.
\end{itemize}
\end{defn}

\begin{defn}
A theory is sequential if it directly interprets $\AS$.
\end{defn}

In $\AS$, and therefore in any sequential theory, we can give basic set-theoretic definitions, in particular, definitions of sequences of any given finite length. This allows us to focus solely on one-dimensional interpretations, as provided by the proposition below.

\begin{fact}
For any sequential theory $S$ and an interpretation $\mathtt{t}$ of some theory in $S$, there exists a $1$-dimensional interpretation $\mathtt{t}^{\ast}$ that is isomorphic to $\mathtt{t}$.\footnote{Note that $\mathtt{t}^{\ast}$ may be relativized.}
\end{fact}

Furthermore, the choice of a direct interpretation of AS is irrelevant for our purposes, as guaranteed by the following proposition -- the fact we keep in mind at all times without mentioning it explicitly.

\begin{fact}
For any sequential theory $T$ and any direct interpretations $\mathtt{a}_{1}$, $\mathtt{a}_{2}$ of $\AS$ in $T$, there exist binary $\Lang[T]$-formulas $\iota$ and $\rho$ such that:
\begin{itemize}
\smallskip
\item $T$ proves that $\iota$ determines an isomorphism from a successor-closed initial segment of the class of finite ordinals of $\mathtt{a}_{1}$ onto an initial segment of the class of finite ordinals of $\mathtt{a}_{2}$;
\smallskip\smallskip
\item for any natural number $n$, $T$ proves that:
\begin{itemize}
    \item for any $x$ there is $y$ such that $\rho(x,y)$;
    \smallskip
    \item for any $y$ there is $x$ such that $\rho(x,y)$;
    \smallskip
    \item for any $\mathtt{a}_{1}$-sequence $s_{1}$ and $\mathtt{a}_{2}$-sequence $s_{2}$, both of length $n$,
    
    $\rho(s_{1},s_{2})$ holds if and only if for any $k$ such that $(k\ls n)^{a_{1}}$,
\end{itemize}
\end{itemize}
\begin{equation*}
(\forall z)(s_{1}(k)=z \leftrightarrow(\forall k')(\iota(k,k')\land s_{2}(k')=z))\text{.}
\end{equation*}
\end{fact}
\bigskip
We close this section with a brief overview of some other notation.
\begin{itemize}
\smallskip
\item[] $\Def{\calM}$ denotes the set of all subsets of $\calM$ that are definable in $\calM$ with parameters.
\smallskip\smallskip
\item[] $\Defpf{\calM}$ denotes the set of all subsets of $\calM$ that are paraeter-free definable in~$\calM$.
\smallskip\smallskip
\item[] $\bbN$ denotes the standard model of $\PA$.
\smallskip\smallskip
\item[] $\omega$ denotes the set of all natural numbers.
\smallskip\smallskip
\item[] $\ZFbase$ denotes $\ZF$ with all instances of the replacement and comprehension schemes removed.
\smallskip\smallskip
\item[] $\simeq$ denotes an isomorphism relation determined by the context.
\end{itemize}


\section{Ground work}

Let us recall that our attention is restricted to finite, purely relational languages, with function symbols used only for notational convenience. Furthermore, since our primary interest lies in sequential theories, all definitions are formulated with this in mind.

\subsection{Schemes}

\begin{defn}
An $\Lang$-scheme is a first-order sentence of the language $\Lang\cup\cubr{P}$, where $P$ is a unary relational symbol absent from~$\Lang$.
\end{defn}

If $\sigma$ is an $\Lang$-scheme, then $\Lang$ is called the language of $\sigma$, and is denoted by $\Lang[\sigma]$. We write $\sigma(P)$ to indicate that the fresh symbol is~$P$. 

\begin{defn}
An instance of a scheme $\sigma(P)$ is a formula of the form
\begin{equation*}
(\forall\tup{y})(\sigma[\varphi(x,\tup{y})/P(x)])\text{,}    
\end{equation*}
where $\varphi(x,\tup{y})$ is in $\Lang_{\sigma}$. If the tuple $\tup{y}$ is of zero length, the instance is said to be parameter-free.
\end{defn}

The set of all instances of $\sigma$ is denoted by $\sigma(\Lang_{\sigma})$, and the set of all parameter-free instances by $\sigma(\Lang_{\sigma}^{\mathrm{pf}})$. It turns out that for every r.e., sequential theory $T$, there exists an $\Lang[T]$-scheme $\tau$ such that $T$ is axiomatizable by both $\tau(\Lang[T])$ and $\tau(\Lang[T]^{\mathrm{pf}})$; see \cite{Visser_Vaught}.

\bigskip
A trivial example of a scheme is a sentence; the set of all (parameter-free) instances of such a scheme consists solely of that sentence. More notable are the arithmetical induction scheme, denoted by $\Ind(P)$, i.e.
\begin{equation*}
\bigl(P(0)\land(\forall x)(P(x)\rightarrow P(x+1))\ \rightarrow\ (\forall x)(P(x))\bigr)\text{,}
\end{equation*}
and the comprehension scheme $\Com(P)$:
\begin{equation*}
(\exists Y)(\forall x)(x\in Y\leftrightarrow P(x)).\footnote{Let us emphasise that this is indeed a scheme, since we regard second-order theories as being encoded by first-order theories; so the above is, in fact, \mbox{$(\exists y \in \mathrm{Class})(\forall x \in \mathrm{Obj})(x \in y \leftrightarrow P(x))$.}}
\end{equation*}
The consequences of $\PAm+\Ind(\Lang[\PA])$ coincide with $\PA$, and the same is true for $\PAm+\Ind(\Lang[\PA]^{\mathrm{pf}})$; see \cite[Section 8.1, Exercise 8.3]{kaye1991book}. The set of all instances of the scheme consisting of $\Com(\Lang[\Zt])$ together with $\PAm$ and the induction axiom
\begin{equation*}
(\forall X)(0\in X\,\,\land\,\,(\forall y)(y\in X\rightarrow y+1\in X)\ \,\rightarrow\ \,(\forall y)(y\in X))
\end{equation*}
axiomatizes $\Zt$. However, the set of all parameter-free instances of this scheme is strictly weaker -- even when the induction axiom is replaced by the (parametric) induction scheme and supplemented with certain further strengthenings -- as shown in \cite{Kanovei_pfz2}.

\bigskip

Less known is the arithmetical Tarski scheme, denoted by $\Tarski{\Lang[\PA]}$. It is defined as $\lnot(\SAT{\Lang[\PA]})$, where $\SAT{\Lang[\PA]}$ is the conjunction of the compositional clauses that describe the satisfaction predicate $S$ for the language of arithmetic -- in a manner that is meaningful within a base system; in our case, $\mathrm{I}\Delta_{0}+\exp$ is sufficient. Concretely, $\SAT{\Lang[\PA]}$ consists of the $\forall \langle\varphi(\tup{x}),\tup{c}\rangle$-closures, where $\varphi\in\mathrm{Form}_{\Lang[\PA]}$, of the following formulas:
\begin{align*}
\varphi(\tup{x})&\in\mathrm{Atom}_{\Lang[\PA]}&\hspace{0.2em}&\rightarrow\hspace{0.3em}&\bigl(S(\varphi(\tup{x}),\tup{c})&\leftrightarrow\Sat_{\mathrm{Atom}_{\Lang[\PA]}}(\varphi(\tup{x}),\tup{c})\bigr)\text{;}\footnotemark\\
\varphi(\tup{x})&=(\exists y)\psi(\tup{x}y)&\hspace{0.2em}&\rightarrow\hspace{0.3em}&\bigl(S(\varphi(\tup{x}),\tup{c})&\leftrightarrow(\exists d)S(\psi(\tup{x}y),\tup{c}d)\bigr)\text{;}\\
\varphi(\tup{x})&=\varphi_1(\tup{x})\land\varphi_2(\tup{x})&\hspace{0.2em}&\rightarrow\hspace{0.3em}&\bigl(S(\varphi(\tup{x}),\tup{c})&\leftrightarrow S(\varphi_1(\tup{x}),\tup{c})\land S(\varphi_2(\tup{x}),\tup{c})\bigr)\text{;}\\
\varphi(\tup{x})&=\lnot\psi(\tup{x})&\hspace{0.2em}&\rightarrow\hspace{0.3em}&\bigl(S(\varphi(\tup{x}),\tup{c})&\leftrightarrow\lnot S(\psi(\tup{x}),\tup{c})\bigr)\text{.}
\end{align*}
\footnotetext{$\Sat_{\mathrm{Atom}_{\Lang[\PA]}}$ is a satisfaction predicate for atomic formulas of $\Lang[\PA]$. Since we take $\Lang[\PA]$ to be a relational language and $\mathrm{I}\Delta_{0}+\exp$ to be the base theory, its construction is more straightforward.}This scheme can be seen as expressing, relative to some previously chosen base theory (in our case $\mathrm{I}\Delta_{0}+\exp$), that $P$ is not a satisfaction predicate for the language of arithmetic. By Tarski's undefinabilty theorem, $\Tarski{\Lang[\PA]}(\Lang[\PA])$ is provable in $\mathrm{I}\Delta_{0}+\exp$. Later, we also use $\Tarski{\Lang[\ZF]}$ -- a corresponding scheme for the language of set theory.

\subsection{Predicative comprehension}

\begin{defn}
$\PC{\sigma}$ is the extension of the sentence
\begin{equation*}
\forall X\ \sigma[X/P]
\end{equation*}
with the predicative comprehension for $\Lang[\sigma]$, i.e. all sentences~of~the~form:
\begin{equation*}
\forall X\forall x\exists Y\forall z\ \Bigl(z\in Y\leftrightarrow \varphi(z,x,X)\Bigr)\text{,}
\end{equation*}
where $\varphi\in\Lang[\sigma]$.
\end{defn}

Observe that for every sequential scheme $\sigma$ $\PC{\sigma}$ is finitely axiomatizable, see \cite[Theorem 3.3]{Visser_predicative_frege}.

\bigskip

Note also that $\PC{\sigma}$ proves every instance of~$\sigma$. Moreover, if $\calM$ is a model of $\sigma(\Lang[\sigma])$, then $(\calM,\Def{\calM})$ is a model of $\PC{\sigma}$. And therefore $\PC{\sigma}$ is conservative over $\sigma(\Lang[\sigma])$. This is a generalization of the fact that $\ACA_{0}$ is conservative over $\PA$ and the fact that $\GB$ is conservative over $\ZF$. Indeed, $\PC{\PAm+\Ind}=\ACAz$ and $\PC{\ZFbase+\Repl}=\GB$.\todo[noline]{Trzeba gdzieś wprowadzić $\Repl$.}\todo[noline]{Może chcemy gdzieś wytłumaczyć, że suma schematów jest schematem? A może nie...}

\bigskip

If $S$ is a set of schemes in a given language, then we use $\PC{S}$ to denote $\bigcup_{\sigma\in S}\PC{\sigma}$. In particular, if $T$ is a theory, then $\PC{T}$ denotes $\PC{\emptyset}+T$.

\begin{defn}
$\sPC{\sigma}$ is the extension of the sentence
\begin{equation*}
\forall X\ \sigma[X/P]
\end{equation*}
with the set-parameters-only predicative comprehension for $\Lang[\sigma]$, i.e. all sentences of the form:
\begin{equation*}
\forall X\exists Y\forall z\ \Bigl(z\in Y\leftrightarrow \varphi(z,X)\Bigr)\text{,}
\end{equation*}
where $\varphi\in\Lang[\sigma]$.
\end{defn}

As before, if $S$ is a set of schemes in a given language, then we use $\sPC{S}$ to denote $\bigcup_{\sigma\in S}\sPC{\sigma}$. Similarly, if $T$ is a theory, then $\sPC{T}$ denotes $\sPC{\emptyset}+T$.

\bigskip

Let us note that $\sPC{\sigma}$ proves every parameter-free instance of~$\sigma$. In fact, it is conservative over $\sigma(\Lang[\sigma]^{\mathrm{pf}})$, since if $\calM$ is a model of $\sigma(\Lang[\sigma]^{\mathrm{pf}})$, then $(\calM,\Defpf(\calM))$ is a model of $\sPC{\sigma}$. It is worth noting too that $\PC{\sigma}$ extends $\sPC{\sigma}$.

\bigskip
Unlike in the case of $\mathrm{PC}$, $\sPC{\sigma}$ may not be finitely axiomatizable; for example, this occurs for $\sPC{\PAm+\Ind}$. This theory is in fact reflexive, i.e. it proves the consistency of each of its finite fragments. This can be shown by first $\omega$-interpreting $\sPC{\PAm+\Ind}$ in a truth theory $\mathsf{UTB}$ and then checking that $\PA$ proves the consistency of each finite part of the latter theory. See \cite{fischer_minimal_truth} for the details.

\bigskip
The above two definitions serve as a basis for our investigation of schemes. Let us recall that the intended interpretation of a scheme is an open-ended description of properties of intended models or an intended model -- that is a description that remains valid after placing it in a further first-order context. From this perspective, $\sPC{\sigma}$ describes properties of any potential expansion of an intended model arising from placing of the model in a potential further context. The difference between the intended interpretations of $\mathrm{PC}$ and $\mathrm{sPC}$ lies in the approach to undefinable entities of the object-sort; we will not address this issue in the present article. However we note that some of our arguments, in particular the verification of definiteness properties for set theory or comprehension, actually require the usage of $\mathrm{PC}$ and we do not know how to carry them in pure $\mathrm{sPC}$. In contrast, the categoricity arguments for various extensions of the first-order arithmetic can be given in $\mathrm{sPC}$ alone.

\bigskip
Having the intended interpretation in mind, we give the following definitions of two scheme-template counterparts of the provability predicate and corresponding equivalence relations.

\begin{defn}
\ 

\textbf{a.} $\sigma_{1}$ (standalone) implies $\sigma_{2}$ if
\begin{equation*}
\mathrm{(s)PC}[\emptyset]\proves(\forall X)\sigma_{1}[X/P]\rightarrow(\forall X)\sigma_{2}[X/P]\text{.}
\end{equation*}

\textbf{b.} $\sigma_{1}$ and $\sigma_{2}$ are (standalone) equivalent if
\begin{equation*}
\mathrm{(s)PC}[\emptyset]\proves(\forall X)\sigma_{1}[X/P]\leftrightarrow(\forall X)\sigma_{2}[X/P]\text{.}
\end{equation*}
\end{defn}

All the subsequent considerations would remain valid if a scheme were replaced with another that is standalone equivalent to it. For instance, $\PAm+\Ind(P)$ is standalone equivalent to the scheme expressing the least number principle
\begin{equation*}
\PAm\ \ +\ \ (\exists x)P(x)\rightarrow(\exists x)(\forall y<x)P(x)\land\lnot P(y)\text{.}
\end{equation*}

\bigskip
Above, we used the adverb ``standalone'' in the names of variants based on $\mathrm{sPC}$ as opposed to $\mathrm{PC}$. We will adhere to this convention in what follows.

\subsection{Strong models and interpretations}
If $\calM$ is a model for $\Lang[\sigma]$ and $Y$ is a subset of its universe, then we write $(\calM,Y)$ to denote the expansion of $\calM$ to a model of $\Lang[\sigma]\cup\{P\}$ in which $P$ is interpreted as $Y$. Using this convention, the following definition is formulated within $\sPC{\AS}$ and can therefore be used in $\sPC{\sigma}$ for any sequential $\sigma$.

\begin{defn}
Given a fixed language $\Lang=\cubr{S_{1},\ldots,S_{k}}$ and a fixed $\Lang$-scheme $\tau$, while working in $\sPC{\AS}$, we say that
\begin{equation*}
\calM=\robr{X_{\dom}, X_{=}, X_{1}, \ldots, X_{k}}    
\end{equation*}
is a strong model of $\tau$, denoted by $\calM\strongs\tau$, if for every $X_{=}$-invariant relation $Y$ on $X_{\dom}$
\begin{equation*}
(\calM,Y)\models\tau\text{.}
\end{equation*}
\end{defn}

In the above, we do not assume that the full satisfaction predicate for $(\calM,Y)$ is definable; instead, we make use of the fact that $\tau$ is a fixed sentence. On the notational side, if $(\mathcal{M},\mathcal{Y})$ is a model of $\sPC{AS}$, $\mathcal{N}\in\mathcal{Y}$, and $(\mathcal{M},\mathcal{Y})\models(\mathcal{N}\strongs\tau)$, then we write that $\calN$ is a $\mathcal{Y}$-strong model of $\tau$.

\bigskip
A source of concrete examples of strong models arises from interpretations between schemes, as defined below.

\begin{defn}
A (one-dimensional) translation $\mathtt{t}$ of $\Lang[\tau]$ to $\Lang[\sigma]$ is an interpretation of $\tau$ in $\sigma$, denoted by $\sigma\intps\tau$, if $\sPC{\sigma}$ proves that the $\Lang[\tau]$-structure determined by $\mathtt{t}$ is a strong model of~$\tau$.
\end{defn}

A strong model can be seen as a potential interpretation that may arise from interpreting the ground scheme in another. In this sense, interpretations are prior to strong models, with the notion of a strong model being a generalization of that of an interpretation. From the perspective of the correspondence between schemes and open-ended descriptions, an interpretation of a scheme in another relate to a placement of a description in a further context.

\begin{defn}
Interpretations $\mathtt{t}_{1},\mathtt{t}_{2}\colon\sigma\intps\tau$ \textit{are isomorphic}, denoted by $\mathtt{t}_{1}\simeq\mathtt{t}_{2}$, if $\sPC{\sigma}$ proves that the corresponding strong models are isomorphic.
\end{defn}

Observe that if $\mathtt{t}_{1}$ and $\mathtt{t}_{2}$ are isomorphic, then there exists a single formula $\iota(x,y)\in\Lang[\sigma]$ such that $\sPC{\sigma}$ proves that $\iota$ determines an isomorphism between $\mathtt{t}_1$ and $\mathtt{t}_2$. Indeed, suppose that no such formula exists. Then $\sigma(\Lang[\sigma]^{\mathrm{pf}})$ is consistent with the theory 
\begin{equation*}
 \cubr*{\text{``}\iota\text{ is not an isomorphism between }\mathtt{t}_{1}\text{ and }\mathtt{t}_{2}\text{''}}[\iota(x,y)\in\Lang[\sigma]]\text{.}
\end{equation*}
For any model $\calM$ of their union, $(\calM, \Defpf{\calM})$ is a model of $\sPC{\sigma}$. By the definition, no set in $\Defpf{\calM}$ encodes an isomorphism between $\mathtt{t}_1$ and $\mathtt{t}_2$, contradicting the assumption that $\mathtt{t}_{1}$ and $\mathtt{t}_{2}$ are isomorphic.

\bigskip
Before the last definition of this subsection, let us note that scheme interpretations can be composed in the same way as interpretations of theories. Furthermore, for every scheme $\sigma$ there exists an identity interpretation, denoted by $\id_{\sigma}$

\begin{defn}
\ 

\textbf{a.} $\sigma$ is a retract of $\tau$ if there exist $\mathtt{t}\colon\sigma\intps\tau$ and $\mathtt{s}\colon\tau\intps\sigma$ such that $\mathtt{s}\circ\mathtt{t}\simeq\id_{\sigma}$

\textbf{b.} $\sigma$ and $\tau$ are bi-interpretable if there exist $\mathtt{t}\colon\sigma\intps\tau$ and $\mathtt{s}\colon\tau\intps\sigma$ such that $\mathtt{s}\circ\mathtt{t}\simeq\id_{\sigma}$ and $\mathtt{t}\circ\mathtt{s}\simeq\id_{\tau}$.
\end{defn}

We remark that the fact that $\sigma$ is a retract of $\tau$ and $\tau$ is a retract of $\sigma$ does not necessarily imply that $\sigma$ and $\tau$ are bi-interpretable. 

\subsection{\texorpdfstring{$\Phi$}{Φ}-definiteness}

Below, we introduce notions thought of as tools for research on categoricity\-/like, completeness\-/like, and other definiteness-like properties in the first-order realm.
\begin{defn}
Given a formula $\Phi(\calM_1,\calM_2)$ of the second-order language of set theories, the $\Phi$-definiteness statement for $\tau$ is the sentence:
\begin{equation*}
\forall \calM_1,\calM_2\ \Bigl(\bigl(\calM_{1}\strongs\tau\ \land\ \calM_{2}\strongs\tau\bigr)\ \rightarrow\  \Phi(\calM_1,\calM_2)\Bigr)\text{.}
\end{equation*}
\end{defn}

\begin{defn}
We say that $\tau$ is $\Phi$-definite in $\sigma$ if $\PC{\sigma}$ proves the translation into $\Lang[\sigma]$ of the $\Phi$-definiteness statement for~$\tau$.
\end{defn}

Note that in this setting, $\tau$ can be $\Phi$-definite in $\sigma$ also if $\PC{\sigma}$ proves that there is no strong model of $\tau$. This situation does not occur in any example in this article -- unless explicitly stated otherwise -- and we allow ourselves not to signal this explicitly each time.

\bigskip

In general, whether $\tau$ is $\Phi$-definite in $\sigma$ may depend on which direct interpretation of $\AS$ is used to translate the $\Phi$-definiteness statement into the language of $\sigma$. However, this is not the case for the notions considered in this article. As a rule, taking into account that in our framework all strong models are one-dimensional and that exact formalizations depend on the choice of interpretations of $\AS$, it seems reasonable, within this framework, to restrict attention to $\Phi$ that are closed under isomorphism; that is, to those $\Phi$ for which $\mathrm{sPC}[\AS]$ proves that for any $\calM_{1}$, $\calM_{2}$, $\calM_{1}'$, and $\calM_{2}'$ with $\calM_{1}\simeq\calM_{1}'$ and $\calM_{2}\simeq\calM_{2}'$, the translation of $\Phi(\calM_{1},\calM_{2})$ holds if and only if $\Phi(\calM_{1}',\calM_{2}')$ holds.

\begin{defn}
\ \\
\textbf{a.} $\iso(\calM_{1},\calM_{2})$ is the formula\footnote{The exact syntactic forms of $\iso$ and $\eeq{\alpha}$ depend on the languages of the models taken as free variables.}
\begin{equation*}
\exists F\ F\colon\calM_{1}\isoto\calM_{2}\text{.}
\end{equation*}
\textbf{b.} For a given sentence $\alpha$, $\eeq{\alpha}(\calM_{1},\calM_{2})$ is the formula
\begin{equation*}
\calM_{1}\models\alpha\ \leftrightarrow\ \calM_{2}\models\alpha\text{.}
\end{equation*}
\textbf{c.} We say that $\tau$ is categorical in $\sigma$ if it is $\iso$-definite in $\sigma$.

\smallskip
\textbf{d.} $\tau$ is complete in $\sigma$ if for every sentence $\alpha\in\Lang[\tau]$, $\tau$ is $\eeq{\alpha}$-definite~in~$\sigma$.
\end{defn}

One of the differences between the notion of categoricity for schemes just introduced and its two predecessors -- second-order internal categoricity from \cite{Vaananen_Wang_Notre_Dame} and \cite{Button_Walsh_book}, and strong internal categoricity from \cite{ena_lel} -- is that the above definition introduces an additional parameter, $\sigma$, which we interpret as an external context (or referee, or judge) for determining whether the given scheme is categorical. Categoricity, being a metatheoretical property, depends on the metatheory in which it is investigated. As shown in \cite[Appendix 10.C]{Button_Walsh_book}, the $\iso$-definiteness statement\footnote{More precisely, its natural counterpart that deals directly with $n$-ary relations, instead of treating them as sets of ordered pairs.} for the induction scheme is not provable in predicative comprehension\footnote{That is $\PC{\emptyset}$.}. In the definition given in \cite{ena_lel}, the judge is simply a disjoint copy of the scheme under consideration -- this is both philosophically problematic and formally suboptimal. It seems that in many important cases -- such as those of $\PA$, $\Zt$, and $\mathrm{KF}\mu$ -- everything required of the judge can already be provided by $\AS$ alone.

\bigskip

Now, let us turn to our analogue of quasi-categoricity.

\begin{defn}
Given a formula $\Phi(\calM_1,\calM_2)$, $\Phi$-in-every-cardinality is the formula:
\begin{equation*}
((\exists F)\,F\colon\calM_{1}\bijto\calM_{2})\ \rightarrow\ \Phi(\calM_{1},\calM_{2})\text{.}
\end{equation*}
\end{defn}

\begin{defn}
We say that $\tau$ is $\Phi$-definite in every cardinality in $\sigma$ if $\tau$ is $(\Phi\text{-in-every-cardinality})$-definite in $\sigma$.
\end{defn}

\begin{defn}
\ 

\textbf{a.} $\tau$ is categorical in every cardinality in $\sigma$ if it is $\iso$-definite in every cardinality in $\sigma$;

\textbf{b.} $\tau$ is complete in every cardinality in $\sigma$ if for every sentence $\alpha\in\Lang[\tau]$, $\tau$ is $\eeq{\alpha}$-definite in every cardinality in $\sigma$;
\end{defn}

All definitions introduced in this subsection can also be reformulated by replacing $\mathrm{PC}$ with $\mathrm{sPC}$. We use the adverb ``standalone'' in the names of variants based on $\mathrm{sPC}$. For example, we say that $\tau$ is standalone $\Phi$-definite in $\sigma$ if $\sPC{\sigma}$ proves the $\Phi$-definiteness statement for~$\tau$; similarly, $\tau$ is standalone categorical in $\sigma$ if it is standalone $\iso$-definite in~$\sigma$.

\bigskip

At this point, it is worth noting an important distinction between $\tau$ not being (standalone) $\Phi$-definite in $\sigma$ and $\tau$ being (standalone) $\Phi$-non-definite in $\sigma$. In the first case, $\mathrm{(s)PC}[\sigma]$ does not prove the $\Phi$-definiteness statement for $\tau$; in the second, $\mathrm{(s)PC}[\sigma]$ proves its negation. We apply similar distinctions to other notions as well -- for instance, we say that $\tau$ is standalone non-categorical in $\sigma$ if $\sPC{\sigma}$ proves the negation of the $\iso$-definiteness statement for $\tau$.


\section{Basic results}

This section is devoted to outlining the main properties of various definiteness notions introduced in the previous section. In particular, in Section 3.1 we show that these notions interact well with the newly introduced notions of retract and bi-interpretability notions for schemes: in the sense to be explained below, each $\Phi$-definiteness notion introduced above is closed under taking retracts, both with respect to the schemes which are being tested for $\Phi$-definiteness, and the collection of schemes which provide us with the metatheory. This witnesses the robustness of our notions. Then, in the next subsection 3.2, we construct schemes that separate various definiteness criteria, along all the three key dimensions constituting our conceptual environment: the quasi-categoricity/categoricity, the completeness/categoricity and the standalone/involving parameters distinctions. Section 3.3 is devoted to the verification which operations on schemes preserve their $\Phi$-definiteness, for various possible $\Phi$'s.

\subsection{Closure properties}

We begin this section with a lemma showing that, when proving categoricity in $\mathsf{AS}$, we may assume that the strong models under consideration inherit the equality relation from the background model of $\mathsf{AS}$. This assumption will simplify some of the subsequent arguments. We note that the proof below goes through also for any other of the $\Phi$-definiteness notions studied in this paper, in the place of ``categorical''.
\begin{lem}\label{lem_abs_equal}
Suppose that $\tau$ is (standalone) categorical in $\mathsf{AS}$ with respect to strong models with absolute equality. Then $\tau$ is (standalone) categorical in $\mathsf{AS}$.
\end{lem}

\begin{proof}
We give a proof for the standalone version. Fix $(\mathcal{M},\mathcal{X})\models \sPC{\mathsf{AS}}$ and suppose that $\mathcal{N}_1, \mathcal{N}_2$ are $\mathcal{X}$-strong models of $\tau$. Let $=^i$ denote the equality relation on $\mathcal{N}_i$. Recall that (see \cite[Section 2.4]{visser_sivs}) we have a one-dimensional interpretation $\mathsf{Num}: \mathsf{AS}\rhd \PAm$. Let $x=0,x=1,x=2$ denote the predicates defining the equivalence classes of $0$, $1$ and $2$ under this interpretation. We define a two-dimensional interpretation of $\mathsf{AS}$, whose domain, denoted with $M^3$, consists of three types of elements: ($i$) sets, ($ii$) 1-atoms and ($iii$) 2-atoms. The elements of $M^3$ will be encoded as pairs $(a_1,a_2)$ with elements of type (1) consisting of these pairs in which $a_2=0 $, the elements of type (2) consisting of those pairs in which $a_2= 1$ and the elements of type (3) of those pairs such that $a_2=2$. The idea is that within each type, each element of $\mathcal{M}$ is represented via a possibly infinite family of objects, corresponding to different elements of the equivalence classes of $0,1,2$ (note that this can be simplified if $\mathcal{M}$ contains at least three definable elements or we allow for parametric interpretability). On so defined universe we introduce the equality relation $=^3$ given by $(a_1,a_2) =^3 (b_1,b_2)$ iff both tuples belong to the same type and
\begin{enumerate}
    \item either $(a_1,a_2)$ belongs to type $(1)$ and $a_1=b_1$; or
    \item $(a_1,a_2)$ belongs to type $(2)$ and $a_1=^1 b_1$; or
    \item $(a_1,a_2)$ belongs to type (3) and $a_1=^2 b_1.$
\end{enumerate}
(recall that $=^i$ is the equality relation of $\mathcal{N}_i$; we tacitly assume that for $i=1,2$ $=^i$ is defined on the whole universe and coincides with identity outside of $\mathcal{N}_i$). Finally we define the relation $\in^3$: $(a_1,a_2) \in^3 (b_1,b_2)$ iff both of the following hold
\begin{enumerate}
    \item $b_2=0$ and
    \item there are $c_1, c_2$ such that $(a_1,a_2)=^3(c_1, c_2)$ and $\tuple{c_1,c_2}\in b_1$
\end{enumerate}
 In the above $\tuple{c_1, c_2}$ stands for an arbitrary element $z$ coding the tuple $\tuple{c_1, c_2}$ , where $\tuple{\cdot,\cdot}$ denotes the canonical Kuratowski pair (defined within $\mathsf{AS}$). Note that there can be more than one such element, so $\tuple{c_1,c_2}\in b_1$ should be read as "there is $z$ such that $z = \tuple{c_1,c_2}$ and $z\in b_1$". We check that $(M^3, =^3, \in^3)\models \mathsf{AS}$. The emptyset axiom is clear. To check the adjunction axiom fix any $x:=(a_1,a_2),y:=(b_1,b_2)\in\mathcal{M}^3$. We are looking for $z$ such that $(z = x\cup \{y\})^{\calM^3}$. If $x$ is of type $(2)$ or $(3)$, then any element of the form $(\{\tuple{b_1,b_2}\}, c)$, where $c=0$, can serve as $(x\cup \{y\})^{\calM^3}$. Otherwise, if $x$ is of type (1), then we define $z$ to be any element of the form $(a_1\cup \{\tuple{b_1,b_2}\}, a_2)$ (note that $a_1\cup \{\tuple{b_1,b_2}\}$ is not uniquely determined but we can pick any element which is extensionally equivalent to this set).
 
 Finally we expand the model $(M^3,=^3, \in^3)$ with $(\mathcal{M}, \mathcal{N}_1, \mathcal{N}_2)$-definable subsets carving out an isomorphic copy of $\mathcal{N}_i$ (denoted $\mathsf{N}_i$) within elements of type $i+1$. We show how to do it for $\mathcal{N}_1$, the definitions for $\mathcal{N}_2$ being fully analogous. For example, the universe of $\mathcal{N}_1$ is given via a formula ($\dom^{\mathcal{N}_1}$ denotes the domain of $\mathsf{N}_1$)
 \[\mathrm{N}_1(x_1,x_2):= x_1\in \dom^{\mathcal{N}_1} \wedge x_2=1,\]
 and each $n$-ary relation $R$ used by $\tau$ is defined naturally from $R^{\mathcal{N}_1}$ (the interpretation of $R$ on $\mathcal{N}_1$) using a formula with $2n$ variables. For example, for $n=2$, we use the formula
 \[\mathrm{R}_1(x_1,x_2, y_1,y_2):= \tuple{x_1,y_1}\in R^{\mathcal{N}_1} \wedge x_2 =  y_2=1.\]
 Let $\mathcal{M}^3 := (M^3, =^3, \in^3, \mathsf{N}_1,\mathsf{N}_2)$. We observe that that $(\mathsf{N}_i, =^3)$ is an $\Def{\mathcal{M}^3}$-strong model of $\tau$, because it is isomorphic to $\mathcal{N}_i$ via an $(\mathcal{M},\mathcal{X})$-definable isomorphism and clearly $\Def{\calM^3}\subseteq \mathcal{X}$. Now observe that the quotient model
\[\mathcal{M}^+:= \left(\faktor{\mathcal{M}^3}{=^3}, \faktor{\in^3}{=^3}, \faktor{\mathsf{N}_1}{=^3}, \faktor{\mathsf{N}_2}{=^3}\right),\]
is a model of $\mathsf{AS}$ in the language containing symbols for the universes and relations from $\mathcal{N}_i$'s in which the additional symbols define disjoint isomorphic copies of $\mathcal{N}_i$'s with the absolute equality. We note that $\mathcal{M}^+$ is isomorphic to $\mathcal{M}^3$ by simply quotienting $(\mathcal{M}^3, \in_3, \mathsf{N}_1,\mathsf{N}_2)$ with $=^3$. In particular, $\faktor{\mathsf{N}_i}{=^3}$ is a $\Def{\mathcal{M^+}}$-strong model of $\tau$. Since $(\mathcal{M}^+, \Def{\mathcal{M^+}})\models \sPC{\mathsf{AS}}$, we conclude that there is a formula $\varphi(x,y)$ which defines an isomorphism $\faktor{\mathsf{N}_1}{=^3}\rightarrow \faktor{\mathsf{N}_2}{=^3}$. Using $\varphi(x,y)$ one easily defines an isomorphism $\mathcal{N}_1\rightarrow \mathcal{N}_2$ in $(\mathcal{M}, \mathcal{N}_1, \mathcal{N}_2).$
\end{proof}

We stress that it is guaranteed that the above lemma holds for an arbitrary sequential theory in the place of AS.

\begin{thm}[Closure under retracts]
Fix schemes $\sigma, \tau$ and assume that $\sigma$ is a retract of $\tau$. Let $T$ be any countable collection of schemes. Then
\begin{enumerate}
    \item If $\tau$ is (standalone) categorical in $T$, then $\sigma$ is (standalone) categorical in $T$.
    \item If $\tau$ is (standalone) complete in $T$, then $\sigma$ is (standalone) complete in $A$.
\end{enumerate}
\end{thm}
\begin{proof}
The proofs of both points start in the same way. We give only "standalone" versions of both claims but the proofs straightforwardly adapt. Assume that $\sigma$ is a retract of $\tau$, that is we have interpretations $\mt{t}: \sigma\intps \tau$ and $\mt{s}: \tau\intps\sigma$ such that $\sPC{\sigma}\proves\exists F: V\rightarrow V^{\mt{t}\mt{s}}$. Recall that $V$ denotes the ground structure of $\sPC{\sigma}$. Fix any  $T$, assume that $\tau$ is standalone categorical in $T$ and fix $(\mathcal{M},\mathcal{X})\models\sPC{T}$. Fix any two $\mathcal{X}$-strong models $\calM_{1}\strongs\sigma$, $\calM_2\strongs \sigma$ such that $\calM_i\in \mathcal{X}$. Let $\mathcal{X}_i:= \{X\cap M_i: X\in \mathcal{X}, X \textnormal{ is $=^i$ invariant}\}$, where $=^i$ is the equality on $\mathcal{M}_i$. Observe that $\mathcal{X}_i\subseteq \mathcal{X}$. Since $\calM_i$ is an $\mathcal{X}$-strong model of $\sigma$, $(\calM_i,\mathcal{X}_i)\models \sPC{\sigma}$. Consider $\calM_1^{\mt{t}}$ and $\calM_2^{\mt{t}}$ and observe that, by the definition, $\calM_{i}^{\mt{t}}$ are $\mathcal{X}$-strong models of $\tau.$ $(*)$ We argue for (1): since $\tau$ is standalone categorical in $T$, it follows that in $\mathcal{X}$ there is an isomorphism $F:\calM_1^{\mt{t}}\rightarrow \calM_2^{\mt{t}}.$ Since $\calM_i^{\mt{t}\mt{s}}$ is $\calM_i^{\mt{t}}$ definable, $F$ restricts to an isomorphism $F_{\mt{t}}: \calM_1^{\mt{t}\mt{s}}\rightarrow \calM_2^{\mt{t}\mt{s}}$. However, since $\sigma$ is a retract of $\tau$, there are also isomorphisms $F_i\in \mathcal{X}_i$, $F_i: \calM_i\rightarrow \calM_i^{\mt{t}\mt{s}}$. Hence $F_2^{-1}\circ F_{\alpha}\circ F_1: \calM_1\rightarrow \calM_2$.

Now, starting from $(*)$, we argue for (2). Since $\tau$ is standalone complete in $T$, $\calM_1^{\mt{t}}$ and $\calM_2^{\mt{t}}$ are elementary equivalent. However then it follows that $\calM_1^{\mt{t}\mt{s}}, \calM_2^{\mt{t}\mt{s}}$ are elementary equivalent, because they are obtained from $\calM_1^{\mt{t}}$ and $\calM_2^{\mt{t}}$ via the same (parameter-free) interpretation. However, since $\sigma$ is a retract of $\tau$, there are also isomorphisms $F_i\in \mathcal{X}_i$, $F_i: \calM_i\rightarrow \calM_i^{\mt{t}\mt{s}}$. So it follows, that $\calM_1$ and $\calM_2$ are elementary equivalent.
\end{proof}

\begin{thm}[Dependence on the metatheory]
Fix a scheme $\tau$ and a collection of schemes $T$. Let $S$ be a collection of schemes which is a retract of $T$ (either as theories or as schemes).
\begin{enumerate}
    \item If $\tau$ is (standalone) categorical in $T$, then $\tau$ is (standalone) categorical in $S$.
    \item If $\tau$ is (standalone) complete in $T$, then $\tau$ is (standalone) complete in $S$.
    \item If $\tau$ is (standalone) categorical-in-every-cardinality in $T$, then $\tau$ is (standalone) categorical-in-every-cardinality in $S$.
    \item If $\tau$ is (standalone) complete-in-every-cardinality in $T$, then $\tau$ is (standalone) complete-in-every-cardinality in $S$.
\end{enumerate}
\end{thm}
\begin{proof}

We prove (1) and (3), the argument for the "complete" counterparts being fully analogous. Similarly, we focus on the standalone variants. For $(1)$, assume that $(\mathcal{M},\mathcal{X})\models\sPC{S}$ and $\mt{t}: S\intps T$, $\mt{s}: T\intps S$, $F: V\rightarrow V^{\mt{t}\mt{s}}$ witness the that $S$ is a retract of $T$ (note that $F$ is strictly speaking a multifunction which is invariant under the appropriate equivalence relations). Fix any $\mathcal{N}_0,\mathcal{N}_1\in\mathcal{X}$ such that $\mathcal{N}_i\strongs \tau$. Let $\mathcal{X}^{\mt{t}} = \{X : X\in\mathcal{X}, X\subseteq \dom^\mt{t}, X \textnormal{ is closed under} =^{\mt{t}}\}$. We note that, by predicative comprehension, $\mathcal{X}^{\mt{t}}\subseteq \mathcal{X}$ and the $(\mathcal{M},\mathcal{X})$-definable model $(\mathcal{M}^{\mt{t}}, \mathcal{X}^{\mt{t}})\models \sPC{T}$. Consider $F[\mathcal{N}_i]\subseteq \mathcal{M}^{\mt{t}\mt{s}}\subseteq \mathcal{M}^{\mt{t}}$. Since $F[\mathcal{N}_i]$ is $\mathcal{M}$-definable, $F[\mathcal{N}_i]\in \mathcal{X}$. Since $F[\mathcal{N}_i]$ is invariant under $=^{\mt{s}}$ and $=^{\mt{s}}$ is invariant under $=^{\mt{t}}$ we get that $F[\mathcal{N}_i]\in \mathcal{X}^{\mt{t}}$. We observe that $F[\mathcal{N}_i]$ is an $\mathcal{X}^{\mathtt{t}}$-strong model of $\tau$. Hence there exists an isomorphism $H: F[\mathcal{N}_0]\rightarrow F[\mathcal{N}_1]$ such that $H\in \mathcal{X}^{\mt{t}}.$ Hence $H\in \mathcal{X}$ and $F^{-1}\circ H\circ F\in\mathcal{X}$ is an isomorphism between $\mathcal{N}_0$ and $\mathcal{N}_1$.

The argument for (3) follows the same pattern as the above argument, but alongside with $\mathcal{N}_i$'s we fix also a bijection $G:\mathcal{N}_0\rightarrow \mathcal{N}_1$, $G\in\mathcal{X}$. Then we observe that $F[G]$ is a bijection between $F[\mathcal{N}_0]$ and $F[\mathcal{N}_1]$ which exists in $\mathcal{X}^{\mt{t}}$. So indeed, in $(\mathcal{M}^{\mt{t}}, \mathcal{X}^{\mt{t}})$ there exists an isomorphism $H: F[\mathcal{N}_0]\rightarrow F[\mathcal{N}_1]$, which generates an isomorphism between $\mathcal{N}_0$ and $\mathcal{N}_1$ in $\mathcal{M}$.
\end{proof}

\subsection{Separations}\label{sect:ground_sep}

Let us start by noting that we have the following obvious implications for a scheme $\sigma$ and a collection of schemes $S$:

\begin{enumerate}
    \item If $\sigma$ is standalone $\Phi$-definite (in every cardinality) in $S$, then $\sigma$ is $\Phi$-definite (in every cardinality) in $S$.
    \item If $\sigma$ is (standalone) $\Phi$-definite in $S$, then $\sigma$ is (standalone) $\Phi$-definite-in-every-cardinality in $S$. 
    \item If $\sigma$ is (standalone) categorical in $S$, then $\sigma$ is (standalone) complete in $S$.
    \item If $\sigma$ is (standalone) $\Phi$-definite [in every cardinality] in $S[\mathcal{L}_S]$, then $\sigma$ (standalone) $\Phi$-definite [in every cardinality] in $S$.
\end{enumerate}

Our results in this paper provide counterexamples to invertibility of each of the above implications. To make this long subsection more perspicuous, we isolate further subsections, corresponding in order to the numbers if the implications in the above list.

\subsubsection{With or without parameters?}

For starters, we consider the failure of the reverse implication from (1) and we put it into a more general context. As discussed in the Introduction, our definiteness criteria derive from the notions studied by V\"a\"ananen and Button and Walsh (and much earlier by Feferman and Hellman, \cite{Feferman_Hellman}). One could expect that the definition of strong internal categoricity from \cite{ena_lel} is just a special case of the current one. There is however a subtle difference, which appears already while comparing the two preceding notions of second-order and strong internal categoricity. Namely, just as the predicative comprehension used in this paper, second-order comprehension allows for formulae with both first- and second-order parameters to be substituted for the predicate letter $P$. As a consequence, in order to conclude that a scheme $\tau$ is second-order internally categorical, it is enough to show  that, working in a Henkin model for second-order logic, for any two strong models of $\tau$, there is a second-order parametrically definable isomorphism between them and the parameters involved in the definition of the isomorphism need not be in any way related with the two models. The definition of of strong internal categoricity is more uniform: the formula $\alpha$ from Definition 44(c) may use only the resources of $\mathcal{L}^{\mathsf{duo}+}$. A more precise estimation of proximity between the two notions is given in the next theorem.

\begin{thm}\label{thm_parameters}
\samepage{\ 

\begin{enumerate}
    \item If a scheme $\tau$ is strongly internally categorical, then $\tau$ is categorical in $\tau[\Lang[\tau]]$.
    \item There is a scheme $\tau$ which is not strongly internally categorical, but it is both categorical in $\mathsf{AS}$ and second-order internally categorical.
    \item  If $\tau$ is standalone categorical in $\tau[\mathcal{L}_{\tau}]$, then $\tau$ is strongly internally categorical.
\end{enumerate}
}
\end{thm}
\begin{proof}
For (1) assume that $\tau$ is strongly internally categorical and fix any model $(\mathcal{M},\mathcal{X})\models \PC{\tau[\Lang[\tau]]}$ and any strong models of $\tau$,$\mathcal{N}_0,\mathcal{N}_1\in \mathcal{X}$. Fix any formula $\alpha(x,y)\in\mathcal{L}^{\mathsf{duo}+}$ which defines the isomorphism between the red and the blue model in an arbitrary model of $\tau^{\mathsf{duo}+}$. Observe that $\mathcal{M}'=(\mathcal{M},\mathcal{N}_0,\mathcal{N}_1)$ is a model of $\tau^{\mathsf{duo}+}$. Hence $\alpha$ defines an isomorphism between $\mathcal{N}_0$ and $\mathcal{N}_1$ in $\mathcal{M}'$. Since $\mathcal{N}_0$ and $\mathcal{N}_1$ are parameters from $\mathcal{X}$, the isomorphism defined by $\alpha$ also is a member of $\mathcal{X}$.

For (2), we consider the canonical second order theory of $(\mathbb{Z},\leq)$. For definiteness, our scheme $\tau$ consits of the following formulae
\begin{itemize}
    \item "$\leq$ is a discrete linear order in which every element has a successor and a predecessor."
    \item $\forall x\bigl(P(x) \wedge \forall z\geq z (P(z)\rightarrow P(\mathsf{S}(z)))\rightarrow \forall z(z\geq x\rightarrow P(z)).$
    \item $\forall x\bigl(P(x) \wedge \forall z\leq z (P(z)\rightarrow P(\mathsf{P}(z)))\rightarrow \forall z(z\leq x\rightarrow P(z)).$
\end{itemize}
In the second and third bullet points, the predicates $\mathsf{P}$ and $\mathsf{S}$ express the $\leq$-definable functions of successors and predecessors. We show that $\tau$ is categorical in $\mathsf{AS}$ (hence in any sequential theory). Fix any model of $(\mathcal{M},\mathcal{X})\models \PC{\mathsf{AS}}$ and any $\mathcal{N}_1,\mathcal{N}_2\in\mathcal{X}$ which are strong models of $\tau$ (in the sense of $(\mathcal{M},\mathcal{X})$). Fix $a\in \mathcal{N}_1$ and $b\in\mathcal{N}_2$. By Lemma \ref{lem_abs_equal}, we can assume that both $\mathcal{N}_1$ and $\mathcal{N}_2$ have absolute equality. We say that a set $y$ is a $\tuple{a,b}$ pointed partial isomorphism if 
\begin{enumerate}
    \item $y$ encodes a partial bijection $\mathcal{N}_0\rightarrow \mathcal{N}_1$ and
    \item the domain of $y$ has the $\leq$-least and $\leq$-greatest element and
    \item if $z\in\dom(y)$ is not minimal, then $\mathsf{P}(z)\in\dom(y)$ and for every $w$, $\tuple{z,w}\in y$ iff $\tuple{\mathsf{P}(z),\mathsf{P}(w)}\in y$ and
    \item if $z\in \mathrm{dom}(y)$ is not maximal, then $\mathsf{S}(z)\in \dom(y)$ and for every $w$, $\tuple{z,w}\in y$ iff $\tuple{\mathsf{S}(z),\mathsf{S}(w)}\in y$ and
    \item $\tuple{a,b}\in y$.
\end{enumerate}
Thanks to sequentiality, the above can be written as a $\Lang[\mathsf{AS}]$-formula with parameters $\mathcal{N}_1$, $\mathcal{N}_2$. We remind the reader that $\tuple{\cdot,\cdot}$ is not provably functional in $\AS$, so each time $\varphi(\tuple{x,y})$ occurs in this proof it should be read "There is a $z$ which is \textit{a} Kuratowski-pair of $x,y$ and $\varphi(z)$." (note that, thanks to the adjunction axiom, provably in $\mathsf{AS}$, for every $x,y$ a Kuratowski-pair of $x,y$ exists). Consider the following formula $\varphi(x,y)\in\mathcal{L}_{\in}$ which says "There is a $\tuple{a,b}$-pointed partial isomorphism $z$ such that $\tuple{x,y}\in z$". Since $\varphi(x,y)$ is a $\Lang[\mathsf{AS}]$ formula with parameters, there is $F\in \mathcal{X}$ whose members are precisely the pairs satisfying $\varphi(x,y)$. We claim that $F$ is an isomorphism between $\mathcal{N}_1$ and $\mathcal{N}_2$. The rest of conditions being simpler, we show how to check that $F$ preserves the ordering. So fix an arbitrary $x\in \mathcal{N}_1$ and consider $y=F(x)$. Let $X\subseteq \mathcal{N}_1$ be the set of those $z$'s such that $x\leq z \leftrightarrow F(x)\leq F(z).$ Since $x$ was arbitrary, it is enough to check that $X$ contains all the elements greater than $x$. Then clearly $X$ is nonempty. We shall check that it closed under the successor. So suppose $z\in X$ and fix an $\tuple{a,b}$ pointed isomorphism $y$ such that $\tuple{z,F(z)}\in y$. Assume that $x\leq z$. Then $y'=y\cup \{\tuple{\mathsf{S}(z),\mathsf{S}(F(z))}\}$ is an $\tuple{a,b}$-pointed partial isomorphism and so we have $F(\mathsf{S}(z)) = \mathsf{S}(F(z))$. Since $x\leq z$ ,then $F(x)\leq F(z)$ ($z\in X$) and also $x\leq \mathsf{S}(z)$ and $F(x)\leq \mathsf{S}(F(z)) = F(\mathsf{S}(z))$. If $F(x)\leq F(z)$ we argue analogously. So it follows by induction that for every $z\geq x$, $z\in X$. 

It remains to be shown that $\tau$ is not strongly internally categorical. We shall demonstrate something stronger: we shall show that there is a model $\calM$ of $\ZFC$ with two classes $\calN_1$, $\calN_2$ which are strong models of $\tau$ with respect to all \textit{parametrically definable subsets} of $\calM$ and such that the isomorphism between $\calN_1$ and $\calN_2$ is not parameter-free definable in $(\calM, \calN_1, \calN_2)$. By the classical Ehrenfeucht-Mostowski result (see \cite[Theorem 3.3.10]{Chang_Keisler} there is a model $\mathcal{M}\models\ZFC + V=HOD$ with order indiscernibles $I$ of type $\mathbb{Z}\oplus\mathbb{Z}$ (in the sense of the global well-ordering of the universe). By passing to its elementary submodel, we can assume that every element of $\mathcal{M}$ is definable from indiscernible elements. More precisely, we can replace $\mathcal{M}$ with $K(I)$ (see \cite[Theorem 2.2.2]{Enayat_set_indiscernibles}). Then any permutation of indiscernibles preserving their ordering in the order-type $\mathbb{Z}$ uniquely determines an automorphism of the whole structure $\mathcal{M}$. Let $\mathcal{N}_i$, for $i\in\{1,2\}$ be the subset of $\mathcal{M}$ consisting of the elements of the $i$-th copy of $\mathbb{Z}$-ordered indiscernibles, together with the canonical ordering in the order type of $\mathbb{Z}$. Consider $\mathcal{N} = (\mathcal{M}, \mathcal{N}_1, \mathcal{N}_2)$ and let $\mathcal{X}$ be the set of first-order parameter-free definable subsets of $\mathcal{N}$. Then (by definition), $(\mathcal{N},\mathcal{X})\models \sPC{\ZFC}$. However, consider $\mathcal{N}_1$ and $\mathcal{N}_2$, which are, also by definition, elements of $\mathcal{X}$. Since $\mathcal{N}_i$ has the order type $\mathbb{Z}$, each of them is obviously a strong model of $\tau$. To show that there is no $F\in\mathcal{X}$ it is enough to show that for any $a\in\mathcal{N}_2$ there are $b,c\in \mathcal{N}_1$ such that $(\mathcal{N},a,b)\equiv (\mathcal{N},a,c)$. So fix $a\in\mathcal{N}_2$ and consider the automorphism of the indiscernibles which fixes the second copy of $\mathbb{Z}$ (hence also $\mathcal{N}_2$) and shifts the elements of $\mathcal{N}_1$ by one. By the choice of $\mathcal{M}$ this extends to an automorphism of $\mathcal{M}$ which fixes $\mathcal{N}_1$, $\mathcal{N}_2$ (setwise) and $a$. Hence this automorphism witnesses that $(\mathcal{N},a,b)\equiv(\mathcal{N},a,\mathsf{S}(b))$, for any $b\in\mathcal{N}_1$.

For $(3)$ assume that $\tau$ is not strongly internally categorical and consider the following $\Lang^{\mathsf{duo}+}$-theory
\[\tau^{\mathsf{duo}+} \cup \{\neg \mathsf{ISO}(\alpha(x,y), R,B) : \alpha(x,y)\in\calL^{\mathsf{duo+}}\},\]
where $\mathsf{ISO}(\alpha(x,y), R,B)$ expresses that $\alpha(x,y)$ defines an isomorphism between the red and the blue models. An easy compactness argument show that the above theory is consistent. Let $\mathcal{M}$ be its model and let $\mathcal{X}$ consist of first-order parameter-free definable sets of $\mathcal{M}$. It follows that $(\mathcal{M},\mathcal{X})\models \sPC{\tau[\Lang[\tau]]}$. Moreover the red and the blue models are elements of $\mathcal{X}$ and are both strong models of $\tau$. However, by the definition of $\mathcal{X}$ and the construction of $\mathcal{M}$ no element of $\mathcal{X}$ is an isomorphism between the red and the blue model.
\end{proof}

\subsubsection{Quasi or full categoricity?}

The most prominent example of a scheme which is categorical in every cardinality in a collection of schemes $S$ but not even complete in $S$, is the scheme providing  canonical axiomatization of $\ZF$: $\mathsf{BST} + \Repl(P)$. We leave the discussion of this and similar phemonena to Section 6. Now we provide an example of a scheme which separates standalone categoricity-in-every cardinality from standalone categoricity. 

 Since similarly motivated schemes will occur also in the next sections, let us introduce a handy piece of terminology:

\begin{defn}
Let $T$ be a finite theory in a language $\Lang[T]$. The scheme $\hf{T}$ is formulated in the language $\Lang[T]\cup\{\in,A\}$ (we tacitly assume that $\Lang[T]$ is disjoint from $\{\in,A\}$) and the conjunction of the following schemes:

\begin{enumerate}
    \item the replacement scheme $\Repl(P)$,
    \item the scheme of $\in$-induction,
    \item the following version of the axiom of powerset
    \[\forall x \exists y\bigl(\forall z (z\in y \equiv \neg A(z) \wedge\forall w (w\in z \rightarrow w\in x))\bigr),\]
    \item the following version of the axiom of union
    \[\forall x\exists y \bigl(\neg A(y)\wedge \forall z \bigl(z\in y\equiv \exists w\in x\, z\in x\bigr)\bigr).\]
    \item the axiom saying that for every set $x$ there exists a natural number $n$ and a bijection $f:n\rightarrow x$,
    \item the axiom of extensionality, relativized to elements satisfying $\neg A(x)$.
    \item $\forall x (A(x)\rightarrow \forall y\,\neg y\in x).$
    \item the axioms of $T$ relativised to $A$.
\end{enumerate}  
\end{defn}

We note that it would make sense to define $\hf{T}$ when $T$ is an infinite collection of schemes, but this generality won't be ever used in this paper.

\begin{thm}\label{thm_standalone_quasi_not_full_cat}
There is a scheme $\tau$ such that $\tau$ is standalone categorical in every cardinality in the scheme $\mathsf{BST}+\mathsf{V=HOD}+\Repl(P)$ but $\tau$ is not categorical $\mathsf{BST}+\mathsf{V=HOD}+\Repl(P)$.
\end{thm}
\begin{proof}
Let $\tau$ extend the scheme $\hf{\emptyset}$ with the statement $\forall x\exists a(A(a)\wedge a\notin x)$. Thus the intended models of $\tau$ are simply the hierarchy of infinite sets over an infinite collection of atoms. Clearly already $\ZF$ can prove that there exists two non-isomorphic strong models of $\tau$: we just construct the hierarchies of hereditarily finite sets over atoms of two different cardinalities. We check that $\tau$ is standalone categorical in every cardinality in $\mathsf{BST}+\mathsf{V=HOD}+\Repl(P)$. Fix any $(\mathcal{M},\mathcal{X})\models \sPC{\mathsf{BST}+\mathsf{V=HOD}+\Repl(P)}$, let $\mathcal{N}_1$, $\mathcal{N}_2$ be two $\mathcal{X}$-strong models of $\tau$ and $F\in\mathcal{X}$ be the bijection between $N_1$ and $N_2.$ By the Scott's trick, we can assume that $\mathcal{N}_i$'s have absolute equality (see also the proof of Lemma \ref{lem_comp_mod_of_ZF}) By the use of replacement scheme in $(\mathcal{M},\mathcal{X})$ and $\varepsilon$-induction scheme in $\mathcal{N}_i$ we know that the natural numbers of $\mathcal{M}$ and $\mathcal{N}_i$'s are all isomorphic. Let $A_i$ be the set of atoms in $\mathcal{N}_i$.  By the predicative comprehension in $(\mathcal{M},\mathcal{X})$ we know that $A_i\in \mathcal{X}$. We argue that there is a bijection $G\in\mathcal{X}$, $G:A_1\rightarrow A_2$. If $A_i$ forms a set in $\mathcal{M}$, then so does $N_i$ and actually $N_i$ and $A_i$ have the same cardinality. Since $N_1$ and $N_2$ have the same cardinality, so do $A_1$ and $A_2$. If $A_1$ is a proper class in $\mathcal{M}$, then so does $A_2$ an by the global well-ordering of the universe we can actually define $G$ by transfinite induction along the ordinals. Once we have $G$, we can now extend it via a routine rank recursion. 
\end{proof}

\subsubsection{Categoricity of completeness? Scheme or theory?}

In this section we deal with the last two implications from our initial list. The first two theorems witness the failure of the reversal of the third one, while the next two do the same for the last one. Since we can actually use one and the same scheme in both cases, we decided to group the two points together.

For the rest of this section we work with a scheme $\tau := \hf{\mathrm{DLO}}$, where $\mathrm{DLO}$ is the theory of dense linear orders without endpoints. Let us first observe that the following are provable in $\tau[\Lang[\tau]]$:

\begin{enumerate}
    \item[A] "Every set of atoms is discretely ordered by $<$ and has the least and the greatest element".
    \item[B] (the existence of transitive closures) "For every set $x$ there is the least transitive set $y$ (which is not an atom), such that $x\subseteq y$."
\end{enumerate}
A can be proved by a routine induction on the cardinality of the given $x$ containing only atoms. This is legal by axiom $(5)$ and $\in$-induction. B is a routine argument using $\in$-induction and replacement.

The first two theorems witness that the third implication from our list does not reverse.

\begin{thm}\label{thm:standalone_complete}
$\tau$ is standalone complete in $\mathsf{AS}$
\end{thm}

\begin{proof}

We prove $(1)$. Fix any model $(\mathcal{M},\mathcal{X})\models \sPC{\mathsf{AS}}$ and any strong models $\mathcal{N}_0\in\mathcal{X}$, $\mathcal{N}_1\in\mathcal{X}$ of $\tau$. We observe that by Lemma \ref{lem_abs_equal} we can assume that the identity relations on $\mathcal{N}_i$'s are just the identity in $\mathcal{M}$. The idea of the proof is to define a winning strategy for Duplicator ($\exists$) in the Ehrenfeucht-Fra\"isse game in $\mathcal{N}_1$, $\mathcal{N}_2$. The strategy is roughly as follows: if Spoiler ($\forall$) plays with $x\in \mathcal{N}_i$, first find new atoms in $\mathtt{tc}(x)\cap A_i$ and select some corresponding atoms from $A_{1-i}$ (this mimics Cantor's back-and-forth argument). Then keeping a partial isomorphism between atoms, construct a set $y$ using the set-theoretical structure of $x$. We need some preparations. In what follows the $\mathcal{N}_i$'s predicates will be indexed with a lower index $i$. In particular, $\mathbb{N}_i$ denote the set of ordinals of $\mathcal{N}_i$. Observe that by $\varepsilon$ induction and axiom (5) $\mathbb{N}_i\in \mathcal{X}$ is a strong model of the induction scheme. Since the induction scheme is standalone categorical in $\mathsf{AS}$, there is an isomorphism $F:(\mathbb{N}_0, \in_0)\rightarrow (\mathbb{N}_1, \in_1)$. The following lemma will be useful later on:
\begin{lem}
Suppose that $G\in \mathcal{X}$ is a partial function $G:\mathcal{N}_0\rightarrow \mathcal{N}_1$. Then for every $x\in \mathcal{N}_0$ there is a partial function $f\in\mathcal{M}$ such that $\dom(f)\cap x = \dom(G)\cap x$ and for every $y\in x\cap \dom (f)$, $f(y) = G(y)$.
\end{lem}
\begin{proof}
Fix $G$ and $x$ as above. By axiom $(5)$ we can assume that $x = \{y_0,\ldots, y_{k-1}\}$ for some $k\in \mathbb{N}_0$. Now the existence of $f$ can be proved by $\varepsilon$-induction on $i<k$, using the adjunction axiom.
\end{proof}
Given $x\in \mathcal{N}_0$ and $y\in \mathcal{N}_1$, we shall say that a function $f: \mathsf{tc}_{0}(x)\rightarrow \mathsf{tc}_{1}(y)$ is \emph{good} if $f\in \mathcal{M}$ is an isomorphism of structures
\[f: (\mathsf{tc}_0(x), \in_{0}, A_0, <_0)\rightarrow (\mathsf{tc}_1(y), \in_1, A_1, <_1).\]
We stress that $f$ being good implies that $f$ belongs to the first-order universe of $\mathcal{M}$. We claim that the set $I$ of those pairs $\tuple{\bar{a},\bar{b}}\in \mathcal{N}_{0}^{<\omega}\times \mathcal{N}_1^{<\omega}$ such that there is a good function $f: \mathsf{tc}_0(\{a_1,\ldots, a_n\})\rightarrow \mathsf{tc}_{1}(\{b_1,\ldots, b_n\})$ is a back-and-forth system between $\mathcal{N}_0$ and $\mathcal{N}_1$. Clearly, if $\tuple{\bar{a},\bar{b}}\in I$, then the quantifier free-types of $\bar{a}$ and $\bar{b}$ are the same. We prove the \textit{forth} condition, the \textit{back} condition being fully symmetric. Let $\bar{a}$ abbreviate $a_1,\ldots,a_n$ and $\bar{b}$ abbreviate $b_1,\ldots,b_n$. So assume that $f:\mathsf{tc}_{0}(\{\bar{a}\})\rightarrow\mathsf{tc}_1(\{\bar{b}\})$ is good and consider an arbitrary $a\in \mathcal{N}_0$. Consider $c = A_0\cap_0 \mathsf{tc}_0(a)$ and $d = A_0\cap_0 \mathsf{tc}_0(\{\bar{a}\})$. Since both $c,d\in\mathcal{N}_0$, we know that $\mathcal{N}_0\models ``(c,<), (d,<) \textnormal{ are finite discrete orders}"$. By induction on $x$ in $\mathcal{N}_0$, we prove that for all $x$
\begin{center}
    if $c$ contains exactly $x$ elements, then there is a set $c'\in\mathcal{N}_1$ and an order isomorphism $g:(c\cup d, <_0)\rightarrow (c'\cup f[d], <_1)$.
\end{center}
Note that $f[d]$ is a set in $\mathcal{N}_1$, because it is simply $A_1\cap_1 \mathsf{tc}_1(\{\bar{b}\})$ (because $f$ is an $\in$-isomorphism which preserves atoms). The base step when $c$ is empty is trivial, so assume that $c$ has $x+1$ elements, pick any $e\in c$ and fix $c'$ and an isomorphism $f:(c\setminus \{e\}\cup d, <_0)\rightarrow (c'\cup f[d], <_1)$. By A $c\setminus\{e\}\cup d$ is discretely ordered by $<_0$ and has the least and the greatest elements. Hence there is $i\in \mathbb{N}_0$ such that $e$ is between the $i$-th and $i+1$-st element (for simplicity assume that the first and the last elements of $c\setminus\{e\}\cup d$ are $-\infty$ and $\infty$ respectively.) Call these elements $j,k$. Since $f$ is an isomorphism $f(j)<f(k)$ and by the density of $<_1$ there is $l\in A_1$ such that $f(j)<_1l<_1f(k)$. Let us fix any such $l$ and define $c'' = c'\cup d$ and $f' = f\cup\{\langle e,l\rangle\}$. This concludes the proof of the induction step. 

So let us fix a partial isomorphism $g: (c\cup d, <_0)\rightarrow (g[c\cup d], <_1).$ We say "$x$ is extensionally unique such that $\varphi(x)$" if $\varphi(x)$ holds and any any element satisfying $\varphi$ has the same elements as $x$.  By $\varepsilon$- induction we prove that
\begin{itemize}
    \item[$(*)$] for every set $x$ such that $A_0\cap_0 \mathsf{tc}_0(\bar{a},x) = c\cup d$ there is an extensionally unique good function
    \[h: \mathsf{tc}_{0}(\{\bar{a},x\})\rightarrow \mathsf{tc}_1(\{\bar{b},h(x)\})\]
    extending $f\cup g$.
\end{itemize}

Note that function $h$ in $(*)$ need not be unique up to the identity in $\mathcal{M}$, because of the lack of extensionality. However, because we do have extensionality for sets in $\mathcal{N}_1$, $h$ will be uniquely determined up to extensionality in $\mathcal{M}$.

In the base case $x$ is either an emptyset or an atom. In the former case, we put $h =f\cup g\cup \tuple{\emptyset_0,\emptyset_1}$. In the latter case we put $h=g$. 

So fix a set $y$ and assume that for every $x\in y$ $(*)$ holds. By axiom $(5)$ and the existence of $\in$-isomorphism between $\mathbb{N}_0$ and $\mathbb{N}_1$, $y$ is in bijection with $k\in\mathbb{N}_1$. So let $ \{x_0,\ldots,x_{k-1}\}$ be all the $\in_0$-elements of $y$. Consider the following class function $G\in\mathcal{X}$, $G:k\rightarrow \mathcal{N}_1$, defined $G(i) = h_i(x_i)$,
where $h_i$ is a (extensionally) unique good function \[h_i: \mathsf{tc}_{0}(\{\bar{a},x_i\})\rightarrow \mathsf{tc}_1(\{\bar{b},h_i(x_i)\}).\]
We note that each two $h_i$'s must agree on the intersections of their domains (by uniqueness). By the replacement scheme in $\mathcal{N}_1$ there is a set $z = G[k]$. We define a function $H: \mathsf{tc}_{0}(\{\bar{a},y\})\rightarrow \mathsf{tc}_1(\{\bar{b},z\})$ by putting 
\begin{equation*}
    H(a) = 
    \begin{cases*}
    z & if $a = y$\\
    h_i(a) & if $a\in \mathsf{tc}_0(\{\bar{a},x_i\})$ and $h_i\supseteq f\cup g$ is good. 
    \end{cases*}
\end{equation*}
By the previous comments on the uniqueness of $h_i$, $H$ is indeed a well-defined partial class function $\mathcal{N}_0\rightarrow \mathcal{N}_1$. Hence the restriction of $H$ to $\mathsf{tc}_0(\{\bar{a},y\})$ exists as a first order object in $\mathcal{M}$. This concludes our proof.

\end{proof}

\begin{thm}\label{thm:tau_non-categorical_in_cardinalities_inZF}
$\tau$ is standalone non-(categorical-in-every-cardinality) in $\ZF.$
\end{thm}
\begin{proof}
Consider any model $\mathcal{M}\models \ZF$ and work in $\mathcal{M}$. Let $<_1$ be the ordering on $\mathbb{R}$ isomorphic to $\mathbb{R}\oplus\mathbb{Q}$ and let $<_2 = <_1^{-1}$. For a set $A$ let $\mathbb{HF}(A)$ denote the structure of hereditarily finite subsets of $A$ (in the sense of the background model $\mathcal{M}$). Consider $\mathcal{N}_i = (\mathbb{HF}(\mathbb{R}), \in, <_i,\mathbb{R})$. Then, since we work in $\ZF$, both $\mathcal{N}_i$ are strong models of $\tau$. However they are not isomorphic by the choice of $<_i$'s.
\end{proof}

Two last theorems of this subsection witness that the last implication from our initial list does not reverse.

\begin{thm}
$\tau$ is standalone categorical in the induction scheme but not (categorical\-/in\-/every\-/cardinality in $\PA$).
\end{thm}
\begin{proof}
To prove that $\tau$ is standalone categorical in the induction scheme, fix any model $(\mathcal{M},\mathcal{X})\models \sPC{\tau_{\mathsf{Ind}}}$ and any two strong models $\mathcal{N}_0\calN_1\in\mathcal{X}$ of $\tau$. By repeating the proof of (1) (translating it along a standard direct interpretation of $\mathsf{AS}$ in $\PA$, for example via the Ackermanian membership predicate) we actually obtain a back-and-forth system $I\in\mathcal{X}$. However since now we have induction in $\mathcal{M}$ for all subsets of $\mathcal{M}$ which are in $\mathcal{X}$, we can easily formalize Cantor's back-and-forth argument showing that two countable back-and-forth equivalent structures are isomorphic. So we indeed obtain an isomorphism $F:\mathcal{N}_0\rightarrow \mathcal{N}_1$ such that $F\in\mathcal{X}$. 

Finally we prove that $\tau$ is not (categorical-in-every-cardinality in $\PA$) (treated as a theory). Consider an uncountable model $\mathcal{M}\models \PA$ with a proper initial segment $I$ of size continuum. Let $a>I$ and assume that $<$ is a linear order on $I$ of type $\mathbb{R}\oplus \mathbb{Q}$ (we do not require $<$ to be definable). Let $\mathbb{HF}_a(I)$ denote the subset of $\mathcal{M}$ consisting of elements of $\mathcal{M}$ which are codes of the hierarchy of hereditarily finite sets with elements of $I$ treated as atoms, according to a standard Ackermanian membership predicate $\in_A$, and such that all the sets in $\mathbb{HF}_a(I)$ are greater than $a$ (one can do this e.g. by moving all the sets from $\mathbb{HF}(I)$ by adding $a$). Let $\mathcal{X}$ be the set of the parametrically definable subsets of $(\mathcal{M}, \mathbb{HF}_a(I), I, <)$. Then $(\mathcal{M},\mathcal{X})\models \PC{\PA}$ and both $\mathcal{N}_0 = (\mathbb{HF}_a(I), \in_A, <, I)$ and $\mathcal{N}_1 = (\mathbb{HF}_a(I), \in_A, <^{-1}, I)$ are members of $\mathcal{X}$. Note that both $\mathcal{N}_i$'s are strong models of $\tau$ (because the $\mathbb{HF}_a(I)$ is truly well founded and of height $\omega$). However, $\mathcal{N}_0$ is not isomorphic to $\mathcal{N}_1$, because $<$ is not isomorphic to $<^{-1}$.
\end{proof}

\subsection{Preservation properties}

We shall now deal with the following issue: given a theory $S$, an $S$-definite $\mathcal{L}$-scheme $\sigma(P)$ and an extension $\mathcal{L}'\supseteq \mathcal{L}$ for which $\mathcal{L}'$-schemes $\tau$ is $\sigma(P)\wedge \tau(P)$ $S$-definite as an $\mathcal{L}'$-scheme? As in \cite{ena_lel} we consider two ways of expanding the language $\mathcal{L}$: by adding a new sort for sets of $\sigma$-objects and by adding a fresh predicate. In the former case, we check whether adding full comprehension scheme (for the newly added sort) preserves definiteness. In the latter, we consider a family of minimality and maximality schemes, defined as in \cite{ena_lel}.

\begin{thm}
Suppose that $\sigma$ is categorical in $S$. Then $\sigma\wedge\Com$ is  categorical in $S$ as an $\Lang[\sigma]^{+1}$ scheme.
\end{thm}
\begin{proof}
We shall deal with the parametrized version of the thesis. Fix $S$ and $\sigma$ and suppose that  $\sigma$ is categorical in $S$. Consider any model $(\mathcal{M}, \mathcal{X})\models \PC{S}$ and let $\mathcal{N}_0, \mathcal{N}_1\in\mathcal{X}$ be strong models of $\sigma\wedge\Com$. Let $\mathcal{N}_i'$ be the reduct of $\mathcal{N}_i$ to the object sort. Note that $\mathcal{N}_i'$ is first order definable in the structure $(\mathcal{M},\mathcal{N}_i)$, and hence belongs to $\mathcal{X}$. Similarly, let $\mathcal{X}_i$ be the universe of the set sort in $\mathcal{N}_i$. Then $\mathcal{X}_i\in\mathcal{X}$ (we remind the reader that two-sorted structures are officially first-order structures with the domain partitioned into sets and objects). Since $\sigma$ does not mention any resources related to the set-sort, for any  $Y\in\mathcal{X}$, $(\mathcal{N}_i,Y)\models \sigma$ iff $(\mathcal{N}_i',Y)\models \sigma.$ Hence, $\mathcal{N}_i'$ is a strong model of $\sigma$. So, in $\mathcal{X}$, there is an isomorphism $F:\mathcal{N}_0'\rightarrow \mathcal{N}_1'$. We note that in general, $F$ can be just a binary relation which yields an isomorphism after quotienting with $=^0$ and $=^1$ (the equality relations on $\mathcal{N}_0$ and $\mathcal{N}_1$, respectively): our notation below cover also these cases. We claim that $F'$ can be definably extended to an isomorphism $F':\mathcal{N}_0\rightarrow \mathcal{N}_1$. For this we need to define $F'$ on elements of $\mathcal{X}_1$. Put $F'(x)=y$ iff $F(x)=y$ for an object $x\in N_1$ and for $X\in\mathcal{X}_1$ let
\[F'(X) = Y \iff Y\in\mathcal{X}_2 \wedge \forall y\in N_2'(y\in^2 Y \equiv \exists x \in^1 X F(x) = y).\] 
Clearly so defined $F'$ belongs to $\mathcal{X}$. $F'$ is a relation and $\dom(F') = \mathcal{X}_0$ because $\mathcal{N}_1'$ is a strong model of $\Com$. We note that $F'$ is actually a function which with each set $X\in\mathcal{X}_0$ associates the set of $F$-adjacents of members of $X$.   Since $F$ is an injection, and the equality on $\mathcal{X}_0$ and $\mathcal{X}_1$ is defined by extensionality using $=^0$ and $=^1$ respectively, $F'$ is an injection as well. We check that $F'$ is onto. Fix any $Y\in\mathcal{X}_1$ and consider the set $X\in\mathcal{X}_0$ such that
\[\forall x\in N_1\bigl(x\in^1 X \equiv \exists y\in^2 Y F(x) = y\bigr).\]
So defined $X$ clearly exists, since $\mathcal{N}_0$ is a strong model of $\sigma\wedge \Com$. One readily checks that $F'(X)=Y$, which completes the proof by witnessing that $F'$ is really an isomorphism between $\mathcal{N}_0$ and $\mathcal{N}_1$. 
\end{proof}
Recall that for a formula $\alpha$ with a marked predicate $P$ and any predicate $Q$, $\mu_Q(\alpha)(P)$ denotes the scheme $\alpha\rightarrow \forall x (Q(x)\rightarrow P(x))$ and $\nu_Q(\alpha)(P)$ is the natural dual: $\alpha\rightarrow \forall x (P(x)\rightarrow Q(x))$. Roughly speaking, $\mu_Q(\alpha)(P)$ in a schematic way that $Q$ is a subset of every set which has property $\alpha$. See \cite[Definition 83]{ena_lel}.
\begin{thm}
Suppose that $\sigma(P)$ is (standalone) categorical in $S$ as an $\mathcal{L}_{\sigma}$-scheme. Let $Q$ be any fresh variable not in $\mathcal{L}_{\sigma}\cup\{P\}$. Then $\sigma(P)\wedge \alpha[Q/P]\wedge\mu_Q(\alpha)(P)$ is (standalone) categorical in $S$ as an $\Lang[\sigma]\cup \{Q\}$ scheme.
\end{thm}
\begin{proof}
Fix any $(\mathcal{M},\mathcal{X})\models \PC{S}$ and let $\mathcal{N}_1$, $\mathcal{N}_2$ be any strong models of $\sigma(P)\wedge \alpha[Q/P]\wedge\mu_Q(\alpha)(P)$. Let $\mathcal{N}_i'$ be the reduct of $\mathcal{N}_i$ to $\mathcal{L}_{\sigma}$. Then, $\mathcal{N}_i'$ is a strong model of $\sigma(P)$ and hence there is an isomorphism $F: \mathcal{N}_1\rightarrow \mathcal{N}_2$. We claim that $F$ maps $Q^{\mathcal{N}_1}$ to $Q^{\mathcal{N}_2}$ (which we shall abbreviate as $Q_1$ and $Q_2$ respectively).  Since $F$ is an isomorphism of $\mathcal{L}_{\sigma}$- structures and $\alpha[Q/P]$ is an $\mathcal{L}_{\sigma}\cup\{Q\}$ sentence, then $\mathcal{N}_2\models \alpha[F[Q_1]/P]$ and likewise $\mathcal{N}_1\models \alpha[F^{-1}[Q_2]/P]$. Since both $\mathcal{N}_i$'s are strong models of $\mu_Q(\alpha)$ we conclude that $\mathcal{N}_1\models Q\subseteq F^{-1}[Q_2]$ and $\mathcal{N}_2\models Q\subseteq F[Q_1]$. Hence $Q_1\subseteq F^{-1}[Q_2]$ and $Q_2\subseteq F[Q_1]$. It follows that $Q_2 = F[Q_1]$ and we are done.
\end{proof}

\begin{remark}
Essentially the same argument as above shows that if $\sigma$ is (standalone) categorical\-/in\-/every\-/cardinality in $S$, then $\sigma\wedge \mu_Q(\alpha)$ is (standalone) categorical\-/in\-/every\-/cardinality in $S$ as an $\mathcal{L}_{\sigma}\cup\{Q\}$-scheme.
\end{remark}

Theorems below witness the failure of other natural candidates for preservation properties. The first one shows that the addition of full comprehension scheme need not guarantee the preservation of completeness (in a particularly strong sense).

\begin{thm}
There is a scheme $\tau$ which is standalone complete in $\mathsf{AS}$ but $\tau\wedge \Com$ is (non-complete-in-every-cardinality) in $\ZF$.
\end{thm}
\begin{proof}
Let $\tau$ be the scheme $\mathsf{HF}(DLO)$ from Section \ref{sect:ground_sep}. By Theorem \ref{thm:standalone_complete} $\tau$ is standalone complete in $\mathsf{AS}$, so it suffices to show that provably in $\ZF$ there are strong set-sized models of $\tau\wedge \Com$ which are of the same cardinality but are not elementary equivalent (as two-sorted structures). Work in $\ZF$. Consider the set-sized models $\mathcal{N}_i$ constructed in the proof of Theorem \ref{thm:tau_non-categorical_in_cardinalities_inZF} and let $\mathcal{N}_i^{\mathsf{II}} = (\mathcal{N}_i, \mathcal{P}(N_i)).$ Then clearly $\mathcal{N}_1^{\mathsf{II}}\models \forall x \exists X\subseteq A\bigl(X<x \wedge \forall z\bigl(X<z\rightarrow \exists y (X<y<z)\bigr)$, but this sentence is false in $\mathcal{N}_0^{\mathsf{II}}$.
\end{proof}

\begin{thm}
There is a scheme $\tau$ such that $\tau$ is standalone categorical in every cardinality in the scheme $\mathsf{BST}+\mathsf{V=HOD}+\Repl(P)$ but $\tau\wedge \Com$ is not categorical in every cardinality in $\mathsf{BST}+\mathsf{V=HOD}+\Repl(P)$.
\end{thm}
\begin{proof}
Let $\tau$ be the scheme $\mathsf{HF}(\emptyset) + \forall x\exists a\in A (a\notin x)$ used to prove Theorem \ref{thm_standalone_quasi_not_full_cat}. We showed that $\tau$ is standalone categorical in every cardinality in the scheme $\mathsf{BST}+\mathsf{V=HOD}+\Repl(P)$.

We shall now argue that $\tau\wedge \Com{}$ is not categorical-in-every-cardinality in $\mathsf{BST}+\mathsf{V=HOD} + \Repl(P).$ Pick any model of $\mathcal{M}\models\mathsf{BST}+\mathsf{V=HOD} + \Repl(\mathcal{L}_{\in})$ such that for some ordinals $\kappa<\lambda\in M$ we have $\mathcal{M}\models 2^{\kappa} = 2^{\lambda}$. Working inside $\mathcal{M}$ consider $\mathcal{N}_1 = (\mathbb{HF}(A), \in, \mathcal{P}(\mathbb{HF}(A)))$ and $\mathcal{N}_2 = (\mathbb{HF}(B), \in, \mathcal{P}(\mathbb{HF}(B)))$, where $A$, $B$ are two sets of cardinality $\kappa$ and $\lambda$ respectively and $\mathbb{HF}(C)$ denotes the hiererachy of hereditarily finite sets over atoms from set $C$ (as defined in $\mathcal{M}$). Since $\mathcal{N}_i$'s are definable in $\mathcal{M}$ and their second-order universes are full powersets of their respective first-order universes, in $(\mathcal{M},\Def{M})$ are full models of $\tau\wedge \Com{}$. However, since their first-order universes have different caridnality, they cannot be isomorphic.
\end{proof}


\section{Minimality of Peano Arithmetic}

The question of whether $\PA$ is, in some sense, a minimal theory in its language that, in some sense, sufficiently accurately describes the natural numbers has already been discussed in various contexts. In particular, in \cite{simpson_yokoyama}, it is shown that the canonical sentence expressing the categoricity of the induction scheme entails Weak K\"onig Lemma over a base system $\mathsf{RCA}_0^*$. Working in the same context of reverse mathematics, \cite{kolodziejczyk_yokoyama} shows that any categorical characterization of natural numbers entails the induction principle for $\Sigma_1$ formulae ($\mathrm{I}\Sigma_1^0$). Let us stress that our framework is different as we do not put any restrictions on the complexity of the definitions of sets (in the language of the base theory) but on the other hand we work in much weaker base systems for the object-sort. In particular our weakest system, $\AS$ is much weaker with respect to the consistency strength, then the weakest arithmetical theory typically considered in the reverse mathemathics ($\mathsf{I}\Delta_0 + \exp$). A result closest to ours was obtained by Feferman and Hellman in \cite{Feferman_Hellman}, however, in contrast to our $\mathsf{PC}[\AS]$, their base system (EFSC) is not finitely axiomatizable. The aim of this subsection is to give a glimpse into how this issue appears within our framework. Let us recall that $\PAm+\Ind(\Lang[\PA])=\PA$.

\begin{thm}\label{thm_Ind_s_cat}
$\PAm+\Ind$ is standalone categorical in $\AS$.
\end{thm}
\begin{proof}
By invoking Lemma \ref{lem_abs_equal}, it sufficient to show that $\PAm+\Ind$ is standalone categorical in $\AS$ with respect to strong models with absolute equality. We are working in $\sPC{\AS}$. Let $\calM_{1}$ and $\calM_{2}$ be any two strong models of $\PAm+\Ind$ with absolute equality. Let $\varphi(x_{1},x_{2})$ be the formula asserting that there exists a unique isomorphism (of $\Lang[\PA]$-structures) from $\calM_{1}\restr{\leqs x_{1}}$ onto $\calM_{2}\restr{\leqs x_{2}}$ -- where the notion of isomorphism is formalized in $\AS$, in an analogous way to the notion of $\tuple{a,b}$-pointed isomorphism from the proof of Theorem \ref{thm_parameters}. Since $\calM_{1}$ and $\calM_{2}$ are both models of $\PAm$, the classes
\begin{equation*}
\{x_{1}\in\calM_{1}\colon\ (\exists x_{2}\in\calM_{2})\varphi(x_{1},x_{2})\}\ \ \text{and}\ \ \{x_{2}\in\calM_{2}\colon\ (\exists x_{1}\in\calM_{1})\varphi(x_{1},x_{2})\}
\end{equation*}
are both non-empty initial segments of $\calM_{1}$ and $\calM_{2}$ respectively, closed under the respective succesors. Thus those classes are equal to $\dom^{\calM_{1}}$ and $\dom^{\calM_{2}}$ respectively, since $\calM_{1}$ and $\calM_{2}$ are both strong models of $\Ind$. Therefore $\varphi(x_{1},x_{2})$ is an isomorphism from $\calM_{1}$ onto~$\calM_{2}$.
\end{proof}

Above, we proved that $\PA$ can be axiomatized by all instances of a scheme that is standalone categorical in $\AS$. The next theorem shows that no finitely axiomatizable subtheory of $\PA$ can be axiomatized by all instances of a scheme that is complete in itself\,\footnote{This theorem has generalizations; for example, it remains true after replacing “finitely axiomatizable subtheory of $\PA$’’ with “$\Sigma_n$-axiomatizable theory consistent with $\PA$’’. To show this, one may in the proof replace $(\bbN,\Def{\bbN})$ with $(\calN,\Def{\calN})$, where $\calN\models\PA+\sigma(\Lang[\PA])$, and replace $\ttM$ with an interpretation that yields a model which is a 
$\Sigma_{n}$-elementary extension of the ground structure; see \cite{GKL}.} (or by all instances of, for example, $\PAm+\Ind$). The idea of the proof goes back to Enayat’s argument that finitely axiomatizable subtheories of $\PA$ are not solid. On the other hand, the theorem one after the next shows that there is a proper subtheory of $\PA$ axiomatizable by all instances of a scheme that is standalone categorical in $\AS$. Its construction parallels the construction of a proper solid subtheory of $\PA$ in \cite{GKL}.

\begin{thm}
There is no scheme $\sigma$ such that $\sigma(\Lang[\PA])$ axiomatizes a finitely axiomatizable subtheory of $\PA$, and $\sigma$ is complete in itself.\footnote{The proof shows that such a scheme is not complete in any scheme whose models include $(\bbN,\Def{\bbN})$.}
\end{thm}
\begin{proof}
Suppose that $\sigma$ is such a scheme. Then $(\bbN,\Def{\bbN})\models\PC{\sigma}$, since $\sigma(\Lang[\PA])$ axiomatizes a subtheory of $\PA$. We will show that there is an $\alpha\in\Lang[\PA]$ such that this model does not satisfy the $\eeq{\alpha}$-definiteness statement for $\sigma$.

\bigskip
Let $n$ be a natural number such that $\sigma(\Lang[\PA])$ is a subtheory of $\IS{n}$. Such a number exists since $\sigma(\Lang[\PA])$ axiomatizes a finitely axiomatizable subtheory of $\PA$. Let $W$ be an extension of $\IS{n}+\lnot\BS{n+1}$ that is bi-interpretable with $\PA$ (see \cite[Theorem 24]{ena_lel}); say via $\ttM\colon\PA\intps W$. Then
\begin{equation*}
\Def{\bbN}\cap\calP(\ttM^{\bbN})=\Def{\ttM^{\bbN}}\text{\ \ and\ \ }\robr{\ttM^{\bbN},\Def{\ttM^{\bbN}}}\models\PC{\sigma}\text{,}
\end{equation*}
since $\ttM^{\bbN}$ is bi-interpretable with $\bbN$ via $\ttM$, and $\ttM^{\bbN}$ is a model of $\sigma(\Lang[\PA])$. Thus $(\bbN,\Def{\bbN})\models(\ttM\strongs\sigma)$. Therefore $(\bbN,\Def{\bbN})$ does not satisfy the $\eeq{\BS{n+1}}$-definiteness statement for $\sigma$.
\end{proof}

\begin{thm}
There exists a scheme $\sigma$ such that $\sigma(\Lang[\PA])\subsetneq\PA$, and $\sigma$ is standalone categorical in $\AS$.
\end{thm}
\begin{proof}
Let $\mathrm{SAT}$ denote the theory $\rmI\Delta_{0}+\SAT{\Lang[\PA]}+\Ind(\Lang[(\SAT{\Lang[\PA]})])$\footnote{In this proof, $\SAT{\Lang[\PA]}$ is treated simply as a sentence, with the predicate $S$ viewed as an ordinary predicate. Also, we slightly abuse the notation, as formally $\Lang[\Ind] = \Lang[\PA]$.}. Let $W$ be an extension of $\IS{n}+\lnot\BS{n+1}$, for some $n>0$, that is bi-interpretable with $\mathrm{SAT}$ (see \cite{GKL}; they use $\mathrm{CT}$ instead of $\mathrm{SAT}$, which is equivalent in this context) -- say via \mbox{$\ttM\colon\mathrm{SAT}\intps W$,} \mbox{$\ttN\colon W\intps\mathrm{SAT}$,} and $i\colon\id\isoto\ttN\circ\ttM$. Without loss of generality, in the ground structure (in $W$), $\dom^{\ttN}$ is a cut and $(0^{\ttN},1^{\ttN},+^{\ttN},\cdot^{\ttN},\leqs^{\ttN})=(0,1,+,\cdot,\leqs)$. Let $\sigma(P)$ be the scheme being the conjunction of the following:
\begin{itemize}
\item $\IS{n}+\neg\BS{n+1}$,
\item ``$\dom^{\ttN}$ is a cut'',
\item $\ttN\models\SAT{\Lang[\PA]}+\rmI\Delta_{0}+\exp$,
\item ``if $P$ is a cut, then $\dom^{\ttN}\subseteq P$'',
\item $i\colon\id\isoto\ttM\circ\ttN$.
\end{itemize}
Let us highlight that all the above can be stated as single sentences. We claim that the scheme $\tau$ defined as
\begin{equation*}
\IS{n}\ \ \land\ \ \bigl(\lnot\BS{n+1}\rightarrow\sigma\bigr)\ \ \land\ \ \bigl(\BS{n+1}\rightarrow\Ind\land\Tarski{\Lang[\PA]}\bigr)
\end{equation*}
is standalone categorical in $\AS$.

\bigskip
We work in $\sPC{\AS}$. Let $\calM_{1}$ and $\calM_{2}$ be any two strong models of~$\sigma$. There are three cases. If both models satisfy $\BS{n+1}$, then both are strong models of $\PAm+\Ind$, and so they are isomorphic (since $\PAm+\Ind$ is standalone categorical in $\AS$). If both satisfy $\lnot\BS{n+1}$, then the arithmetical parts of $\ttN^{\calM_{1}}$ and $\ttN^{\calM_{2}}$ are strong models of $\PAm+\Ind$, and so isomorphic. Thus $(\ttM\circ\ttN)^{\calM_{1}}$ and $(\ttM\circ\ttN)^{\calM_{2}}$ are isomorphic too. Therefore $\calM_{1}$ and $\calM_{2}$ are isomorphic, since $i^{\calM_{1}}$ gives an isomorphism from $\calM_{1}$ onto $(\ttM\circ\ttN)^{\calM_{1}}$, and $i^{\calM_{2}}$ gives an isomorphism from $\calM_{2}$ onto $(\ttM\circ\ttN)^{\calM_{2}}$. The final case remains -- we will show that it cannot occur. Suppose that $\calM_{1}\models\BS{n+1}$ and $\calM_{2}\models\lnot\BS{n+1}$. Then $\calM_{1}$ is isomorphic to the arithmetical part of $\ttN^{\calM_{2}}$, since both are strong models of $\PAm+\Ind$. Thus there exists a satisfaction predicate for $\calM_{1}$, because $\ttN^{\calM_{2}}$ is a model of $\SAT{\Lang[\PA]}$. But $\calM_{1}$ is a strong model of $\Tarski{\Lang[\PA]}$, so there is no satisfaction predicate for $\calM_{1}$. We obtained a contradiction.
\end{proof}


\section{Definiteness of Zermelo-Fraenkel set theory}

Investigations on the definiteness of set theories goes back at least to the work of Zermelo (see \cite{Maddy_Vaananen}). In this chapter, we present several results on this topic -- some are known, others new -- all formulated within our framework.

\begin{thm}
$\ZFbase+\Repl$ is not complete in itself.
\end{thm}
\begin{proof}
Let $\calM$ be a model of $\ZFC+\exists\kappa\ \text{``}\kappa$ is a strongly inaccessible cardinal'' and $\kappa$ be the smallest strongly inaccessible cardinal of $\calM$. Then $(V_{\kappa},\in)^{\calM}$ and $(V,\in)^{\calM}$ are strong models of $\ZFbase+\Repl$ in $(\calM,\Def{\calM})$, and they are not isomorphic, since in the first one there is no strongly inaccessible cardinal, while in the second one there is at least one such a cardinal.
\end{proof}

\begin{lem}\label{lem_comp_mod_of_ZF}
$\PC{\ZFbase+\Repl}$ proves that if $\calM$ is a strong model of $\ZFbase+\Repl$, then either there exists an ordinal number $\kappa$ such that $\calM$ is isomorphic to $(V_{\kappa},\in)$, or $\calM$ is isomorphic to $(V,\in)$.
\end{lem}
\begin{proof}
We will apply a strategy similar to that of Theorem~\ref{thm_Ind_s_cat}. We are working in $\PC{\ZFbase+\Repl}$. Let $\calM$ be a strong model of $\ZFbase+\Repl$. Without loss of generality, we may assume that equality of $\calM$ is absolute. Indeed, this can be justified using Scott's trick: we replace $\dom^{\calM}$ by the class consisting of such $a$'s that
\begin{equation*}
(\exists b\in\dom^\calM)((\forall d)(\rank(d)\ls\rank(b)\rightarrow d\neq^{\calM}b)\ \land\ a=\{c\in V_{\rank(b)}\colon\,c=^{\calM}b\})\text{,}
\end{equation*}
and we modify the $\in^{\calM}$ respectively:
\begin{equation*}
\text{$a$ is in $a'$ \ if and only if \ }(\forall x\in a)(\forall x'\in a')\ x\in^{\calM}x'\text{.}    
\end{equation*}
Let $\iota(x,y)$ be the conjunction of $x\in\Ord$, $y\in\Ord^{\calM}$, and
\begin{equation*}
(\exists s)((\forall z)(z\in s\leftrightarrow(z\in V_{y})^{\calM})\,\land\,(\exists f)\,f\colon\langle V_{x},\in\rangle\isoto\langle s,\in^{\calM}\rangle)\text{.}   
\end{equation*}
Observe that such an $f$ is unique; otherwise $V_{x}$ would admit a non-trivial automorphism, which is impossible since $\PC{\ZFbase+\Repl}\proves\ZF$. Furthermore, if $x\ls x'$, $\iota(x,y)$, and $\iota(x',y')$ hold, then $(y\ls y')^{\calM}$ and $f'\restr{V_{x}}=f$. We will show that the class $\calI=\{y\in\Ord^{\calM}\colon\,(\exists! x)\iota(x,y)\}$ is nonempty and closed under the successor and limits in $\calM$. It then follows that $V^{\calM}$ is isomorphic to the union of a chain of levels of the cumulative hierarchy, and hence it is isomorphic either to some~$V_{\alpha}$~or~to~$V$.

\bigskip
$\calI$ is nonempty because $0^{\calM}$ belongs to it. If $y\in\calI$, then there is exists a function \mbox{$f\colon\langle V_{x},\in\rangle\isoto\langle V^{\calM}_{y},\in^{\calM}\rangle$.} Such an $f$ can be uniquely extended to a function \mbox{$f^{+}\colon\langle V_{x+1},\in\rangle\isoto\langle (V_{y+1})^{\calM},\in^{\calM}\rangle$} by stipulating that $f^{+}(r)=f[\{p\colon\,p\in r\}]$. Then $x+1$ is the unique ordinal $\alpha$ such that $\iota(\alpha,y+1)$ holds, since any isomorphism \mbox{$g\colon\langle V_{\alpha},\in\rangle\isoto\langle (V_{y+1})^{\calM},\in^{\calM}\rangle$} is determined by some \mbox{$g'\colon\langle V_{\beta},\in\rangle\isoto\langle V_{y}^{\calM},\in^{\calM}\rangle$,} with $\beta+1=\alpha$. Therefore $\calI$ is closed under the successor of $\calM$. Now, suppose that any $(z\ls y)^{\calM}$ belongs to $\calI$. Then $\bigcup_{(z\ls y)^{\calM}}f_{z}$ is a class isomorphism from the union of a chain of levels of the cumulative hierarchy onto $\langle V_{y}^{\calM},\in^{\calM}\rangle$. We must to show that this class isomorphism is a set, i.e. we must exclude the possibility that $V$ is isomorphic to some $V_{\alpha}^{\mathcal M}$. Suppose, for the sake of contradiction, that $F$ is an isomorphism from $V$ onto $V_{\alpha}^{\mathcal M}$. Then $F$ induces a class injective map from $V^{\calM}$ into $V_{\alpha}^{\calM}$, since $V^{\calM}$ is a subclass of $V$. Because $\calM$ is a strong model of the replacement scheme, it follows that $V^{\calM}$ is a set in $\calM$. This contradicts the fact that $\calM$ satisfies $\ZF$. Therefore $\calI$ is closed under the limits in $\calM$.

\end{proof}

The next theorem is an application of the preceding lemma. It illustrates the influence of the three components of our framework in the context of definiteness of $\ZF$: the notion of definiteness, the metatheory (or rather, metascheme), and the scheme itself. Note that the set of all instances of $\ZFbase+\Repl+\Tarski{\Lang[\ZF]}$ axiomatizes $\ZF$.

\begin{thm}
\ 

\textbf{a.} $\ZFbase+\Repl$ is categorical in every cardinality in $\AS$.

\textbf{b.} $\ZFbase+\Repl$ is categorical in $\ZFbase+\Repl+$``There is no strongly inaccessible\footnote{Here, $\kappa$ is strongly inaccessible if $(V_{\kappa},\calP(V_{\kappa}))\models\ZFbase+\Repl$; see \cite{blass_et_al} for a discussion of some notions of strongly inaccessible cardinals in the absence of the axiom of choice.} cardinal''.

\textbf{c.} $\ZFbase+\Repl+\Tarski{\Lang[\ZF]}$ is categorical in $\ZFbase+\Repl$.

\end{thm}
\begin{proof}
\ 

a. We begin by invoking Lemma \ref{lem_abs_equal} to allow ourselves to restrict attention to strong models with absolute equality. Let $\frakA=(\calA,\calX)$ be a model of $\PC{\AS}$. Let $\calM_{1}$ and $\calM_{2}$ be strong models of $\ZFbase+\Repl$ with absolute equality, and \mbox{$F\colon\calM_{1}\bijto\calM_{2}$} be a bijection between their universes, all present in $\frakA$. Let
\begin{gather*}
\frakM_{1}=(\calM_{1},\calP(\calM_{1})\cap\calX)\text{\ \ \ and\ \ \ }\calM_{2}^{\ast}=(\dom^{\calM_{1}},F^{-1}[\in^{\calM_{2}}])\text{.}
\end{gather*}
Then $\frakM_{1}$ is a model of $\PC{\ZFbase+\Repl}$ and $\calM_{2}^{\ast}$ is a strong model $\ZFbase+\Repl$ in $\frakM_{1}$. By Lemma \ref{lem_comp_mod_of_ZF}, in $\frakM_{1}$, there exists $G$ such that either $\calM_{2}^{\ast}$ is isomorphic to $(V,\in)$ via $G$, or there is $\kappa$ such that $\calM_{2}^{\ast}$ is isomorphic to $(V_{\kappa},\in)$  via $G$. Suppose that the second case holds. Then, in $\frakM_{1}$, there exists a bijection from $V_{\kappa}$ onto $V$, which contradicts the assumption that $\frakM_1$ is a model of $\PC{\ZFbase+\Repl}$. Thus the first case holds, i.e $\calM_{2}^{\ast}$ is isomorphic to $(V,\in)$ in $\frakM_{1}$. Therefore $\calM_{2}^{\ast}$ is isomorphic to  $\calM_{1}$ in $\frakA$, and so $\calM_{2}$ is isomorphic to $\calM_{1}$ in $\frakA$.

\bigskip
b. We are working in $\mathrm{PC}[\ZFbase+\Repl+$``There is no strongly inaccessible cardinal''$]$. Let $\calM_{1}$ and $\calM_{2}$ be strong models of $\ZFbase+\Repl$. By Lemma \ref{lem_comp_mod_of_ZF}, each of them is isomorphic either to some $(V_{\kappa},\in)$ or to $(V,\in)$. If any of them were isomorphic to some $V_{\kappa}$, then $(V_{\kappa},\calP(V_{\kappa}))$ would be a model of $\ZFbase+\Repl$. Since there is no strongly inaccessible cardinal, this cannot occur. Therefore, each of them is isomorphic to $(V,\in)$, and hence they are isomorphic to each other.

\bigskip
c. We are working in $\PC{\ZFbase+\Repl}$. Let $\calM_{1}$ and $\calM_{2}$ be strong models of $\ZFbase+\Repl+\Tarski{\Lang[\ZF]}$. By Lemma \ref{lem_comp_mod_of_ZF}, each of them is isomorphic either to some \mbox{$(V_{\kappa},\in)$} or to $(V,\in)$. Suppose that one of them is isomorphic to $(V_{\kappa},\in)$. Then this model admits a satisfaction predicate, since for any $\lambda$ there exists a satisfaction predicate for \mbox{$(V_{\lambda},\in)$}. This contradicts the fact that it is a strong model of $\Tarski{\Lang[\ZF]}$. Therefore, both models are isomorphic to $(V,\in)$, and thus to each other.
\end{proof}

The next definition introduces a notion of categoricity specifically tailored to set theories. It allows us to reveal an interesting phenomenon related to the axiom of choice.

\begin{defn}
$\tau$ is height-categorical\footnote{In this case as well, there is a standalone version, but we do not make use of~it.} in $\sigma$ if $\tau$ is $\Phi$-definite in $\sigma$, for $\Phi$ equal to
\begin{equation*}
\bigl((\exists F)\ F\colon\Ord^{\calM_{1}}\isoto\Ord^{\calM_{2}}\bigr)\ \rightarrow\ \iso(\calM_{1},\calM_{2})\text{.}
\end{equation*}
\end{defn}

\begin{thm}
\ 

\textbf{a.} $\ZFbase+\Repl+\AC$ is height-categorical in~$\AS$.

\textbf{b.} $\ZFbase+\Repl$ is not height-categorical in~$\AS$.\footnote{Assuming the consistency of $\ZFC+$``There exists an uncountable well-founded model of $\ZFC$ in the constructible universe''.}

\textbf{c.} $\ZFbase+\Repl$ is height-categorical in~$\ZFbase+\Repl$.
\end{thm}
\begin{proof}
\ 
a. We will follow the standard argument showing that any two transitive models of $\ZFC$ that share the ordinals and have the same subsets of the ordinals are isomorphic (see Theorem 13.28 in \cite{jech}). To simplify the proof, we invoke Lemma~\ref{lem_abs_equal}, which allows us to restrict attention to strong models with absolute equality. On the notational side, given a class $X$ of elements of a model $\calM$ of a set theory, if it does not cause confusion, we write $\calP(X)$ for the class \mbox{$\{y\in\calM\colon\,(\forall z\!\in^{\calM}\!y)\ z\in X\}$.}

\bigskip

We are working in $\PC{\AS}$. Let $\calM_{1}$ and $\calM_{2}$ be strong models of $\ZFbase+\Repl+\AC$ with absolute equalities, and $F$ be an isomorphism from \mbox{$\langle\Ord^{\calM_{1}},\in^{\calM_{1}}\rangle$} onto $\langle\Ord^{\calM_{2}},\in^{\calM_{2}}\rangle$. Let $\widetilde{F}$ denote a natural extension of $F$ to an isomorphism from
\begin{gather*}
\langle\Ord^{\calM_{1}}\cup\calP(\Ord^{\calM_{1}})\cup\calP(\calP(\Ord^{\calM_{1}}))\cup\calP(\calP(\calP(\Ord^{\calM_{1}}))),\in^{\calM_{1}}\rangle\\
\text{onto}\\
\langle\Ord^{\calM_{2}}\cup\calP(\Ord^{\calM_{2}})\cup\calP(\calP(\Ord^{\calM_{2}}))\cup\calP(\calP(\calP(\Ord^{\calM_{2}}))),\in^{\calM_{2}}\rangle\text{.}
\end{gather*}
Such an extension exists because $\calM_{1}$ and $\calM_{2}$ are strong models of $\ZFbase+\Repl$, so for instance, given an $x\in\calM_1$ such that \mbox{$(x\subseteq\Ord)^{\calM_{1}}$,} the class \mbox{$\{F(y)\colon\,(y\in x)^{\calM_{1}}\}$} forms a set in $\calM_{2}$ because all elements of this class are $\calM_{2}$-contained in $F(x)$. Indeed, if $(y\in x)^{\calM_{1}}$, then \mbox{$(F(y)\in F(x))^{\calM_{2}}$,} by the assumption on $F$.

\bigskip

We define an isomorphism $G:\calM_{1}\rightarrow\calM_{2}$ via the formula, with the free-variables $x$ and $y$,
\begin{equation*}
(\exists z)\,(z\subseteq\Ord\times\Ord\land z\simeq\langle\tc(x),\in\rangle)^{\calM_{1}}\land(\widetilde{F}(z)\simeq\langle\tc(y),\in\rangle)^{\calM_{2}}\text{.}
\end{equation*}
We will briefly show that it is well-defined, one-to-one, and onto. Because $\calM_{1}$ is a model of $\ZFC$, for every $x\in\calM_{1}$, there is $z$ such that
\begin{equation*}
(z\subseteq\Ord\times\Ord\land z\simeq\langle\tc(x),\in\rangle)^{\calM_{1}}\text{.}
\end{equation*}
Since $\calM_{1}$ and $\calM_{2}$ both are strong models, and $\widetilde{F}$ is a suitable isomorphism, $\widetilde{F}(z)$ is a well-founded, extensional directed graph in $\calM_{2}$. Thus, in $\calM_{2}$, it is isomorphic to the transitive closure of a uniquely determined set. Moreover, for a fixed $x\in\calM_{1}$, there is only one such a set because if, in $\calM_{1}$, $z$ and $z'$ are both isomorphic to $\langle\tc(x),\in\rangle$, then, in $\calM_{2}$, $\widetilde{F}(z)$ and $\widetilde{F}(z')$ are isomorphic. Therefore, $G$ is well-defined. A similar reasoning, but in the reverse direction, shows that $G$ is onto. Now, suppose that $G(x)=G(x')$. Let $z,z'\in\calM_{1}$ be such that \mbox{$(z\simeq\langle \tc(x),\in\rangle)^{\calM_{1}}$} and \mbox{$(z'\simeq\langle \tc(x'),\in\rangle)^{\calM_{1}}$.} We can conclude that $(\widetilde{F}(z)\simeq\langle\tc(G(x)),\in\rangle)^{\calM_{2}}$ and $(\widetilde{F}(z')\simeq\langle\tc(G(x')),\in\rangle)^{\calM_{2}}$. It follows that \mbox{$(\widetilde{F}(z)\simeq\widetilde{F}(z'))^{\calM_{2}}$}, and hence $(z\simeq z')^{\calM_{1}}$. Therefore \mbox{$(\langle\tc(x),\in\rangle=\langle\tc(x'),\in\rangle)^{\calM_{1}}$.}, and thus $x=x'$. This shows that $G$ is one-to-one.

\bigskip
b. We will show that the second theory mentioned in the statement proves that $\PC{\AS}$ does not prove the (height-categoricity)-definiteness statement for $\ZFbase+\Repl$. From now on, we are working in that theory\footnote{That is $\ZFC+$``There exists an uncountable well-founded model of $\ZFC$ in the constructible universe''.}. Let $\calM$ be an $\omega$-model of $\ZF$ in which the ordinals are pointwise parameter-free definable and which is generated by a sequence of indiscernibles of order type $\langle\bbN\times\bbZ,\ls_{\mathrm{lex}}\rangle$ (see Method 2 on the page 380 in \cite{Enayat_MDO}). Let $\ttN$ be an interpretation in $\calM$ such that, in $\calM$:

\medskip\textbullet\ $\ttN$ is an interpretation of the language $\cubr{\in_{0},V_{1},V_{2},\in_{1},\in_{2},F}$;

\medskip\textbullet\ $(\dom^{\ttN},\in_{0},V_{1},V_{2})$ is a structure of all hereditarily finite sets over a class of atoms grouped into two disjoint classes $V_{1}$ and $V_{2}$;

\medskip\textbullet\ $(\forall x,y)\ (x\in_{1}y\rightarrow x,y\in V_{1})\land(x\in_{2}y\rightarrow x,y\in V_{2})$;

\medskip\textbullet\ $\calM_{1}\coloneq(V_{1},\in_{1})$ and $\calM_{2}\coloneq(V_{2},\in_{2})$ are isomorphic to $(V,\in)$;

\medskip\textbullet\ $F$ is an isomorphism from $\Ord^{\calM_{1}}$ onto $\Ord^{\calM_{2}}$.

\medskip
Then $\frakN=(\ttN^{\calM}\restr{\Lang[\in_{0}]},\Def{\ttN^{\calM}})$ is a model of $\PC{\AS}$. Moreover $\calM_{1}$ and $\calM_{2}$ are strong models of $\ZFbase+\Repl$ in $\frakN$, since they are strong models in
\begin{equation*}
\bigl(\ttN^{\calM}\restr{\Lang[\in_{0}]},\calP(\ttN^{\calM})\cap\Def{\calM}\bigr)
\end{equation*}
(because $\calM_{1}$ and $\calM_{2}$ they are isomorphic to $V$ in $\calM$) and
\begin{equation*}
\Def{\ttN^{\calM}}\ \ \subseteq\ \ \calP(\ttN^{\calM})\cap\Def{\calM}\text{.}
\end{equation*}
Now it suffices to show that $\calM_{1}$ and $\calM_{2}$ are not isomorphic in $\frakN$. Since the second-order part of $\frakN$ is $\Def{\ttN^{\calM}}$, it suffices to show that there is no (parametrically) definable isomorphism from $\calM_{1}$ onto $\calM_{2}$ in $\ttN^{\calM}$. To this end, we will show that for any finite tuple $\overline{c}$ of elements of $\ttN^{\calM}$ there is an automorphism of $\ttN^{\calM}$ that pointwise fixes $\overline{c}$ and $\calM_{1}$ but not $\calM_{2}$. Let $\langle c^{1}_{i}\rangle_{i\in\bbN\times\bbZ}$ and $\langle c^{2}_{i}\rangle_{i\in\bbN\times\bbZ}$ denote the sequences of indiscernibles generating $\calM_{1}$ and $\calM_{2}$ respectively. Then $\ttN^{\calM}$ is generated by $\{c^{j}_{i}\colon\,j\in\{1,2\}\land i\in\bbN\times\bbZ\}$. Thus, without loss of generality, we can assume that $\tup{c}$ consists of the elements of $\{c^{j}_{i}\colon\,j\in\{1,2\}\land i\in I\}$, for some finite $I\subseteq\bbN\times\bbZ$. Because $\calM_{2}$ is generated by the set of indiscernibles of order type $\bbN\times\bbZ$ and $I$ is a finite subset of $\bbN\times\bbZ$, there exists a nontrivial automorphism of $\calM_{2}$ that fixes every element of $\{c^{2}_{i}\colon\,i\in I\}$. Since the ordinals are pointwise parameter-free definable in $\calM_{2}$, this automorphism fixes also every ordinal. We can extend this automorphism to an automorphism of the disjoint union $\calM_{1}\uplus\calM_{2}$ that pointwise fixes $\calM_{1}$, $\tup{c}$, and  $\Ord^{\calM_{2}}$ but moves some elements of $\calM_{2}$. Any automorphism of $\calM_{1}\uplus\calM_{2}$ pointwise fixing $\Ord^{\calM_{1}}$ and $\Ord^{\calM_{2}}$ can be uniquely extended to an automorphism of $\ttN^{\calM}$ -- this is because $\ttN^{\calM}\restr{\Lang[\in_{0},V_{1},V_{2},\in_{1},\in_{2}]}$ is the structure of hereditarily finite sets over atoms from $\calM_{1}\uplus\calM_{2}$ (we use here the assumption that $\calM$ is an $\omega$ model), and $F^{\ttN^{\calM}}$ is an isomorphism from $\Ord^{\calM_{1}}$ onto $\Ord^{\calM_{2}}$. In particular, the automorphism previously obtained can be extended to an automorphism of $\ttN^{\calM}$ that fixes $\calM_{1}$ but not $\calM_{2}$.

\bigskip

c. We are working in $\PC{\ZFbase+\Repl}$. Let $\calM_{1}$ and $\calM_{2}$ be strong models of $\ZFbase+\Repl$, and $F$ be an isomorphism from $\Ord^{\calM_{1}}$ onto $\Ord^{\calM_{2}}$. By Lemma \ref{lem_comp_mod_of_ZF}, $\calM_{1}$ is isomorphic either to some $(V_{\kappa_{1}},\in)$ or to $(V,\in)$, and $\calM_{2}$ is isomorphic either to some $(V_{\kappa_{2}},\in)$ or to $(V,\in)$. Therefore there are four cases. If both are isomorphic to $(V,\in)$, then they are isomorphic to each other. If one of them is isomorphic to $(V,\in)$ and the other is isomorphic to $(V_{\kappa},\in)$, then there is an isomorphism from $\Ord$ onto $\kappa$, and hence $\Ord$ forms a set, which is a contradiction (because $\PC{\ZFbase+\Repl}\proves\ZF$). The last case remains. Suppose that $\calM_{1}$ and $\calM_{2}$ are isomorphic to $(V_{\kappa_{1}},\in)$ and $(V_{\kappa_{2}},\in)$ respectively. By the assumption, $\kappa_{1}$ is isomorphic to $\kappa_{2}$, therefore $\kappa_{1}=\kappa_{2}$, and thus $\calM_{1}$ and $\calM_{2}$ are isomorphic.
\end{proof}



\printbibliography


\end{document}